\documentclass[11pt]{amsart}
\usepackage{amssymb}
\usepackage{verbatim}
\usepackage{hyperref}
\usepackage{color}

\begin{document}
\setcounter{tocdepth}{1}

\newtheorem{theorem}{Theorem}    
\newtheorem{proposition}[theorem]{Proposition}
\newtheorem{conjecture}[theorem]{Conjecture}
\def\theconjecture{\unskip}
\newtheorem{corollary}[theorem]{Corollary}
\newtheorem{lemma}[theorem]{Lemma}
\newtheorem{sublemma}[theorem]{Sublemma}
\newtheorem{fact}[theorem]{Fact}
\newtheorem{observation}[theorem]{Observation}
\theoremstyle{definition}
\newtheorem{definition}{Definition}
\newtheorem{notation}[definition]{Notation}
\newtheorem{remark}[definition]{Remark}
\newtheorem{question}[definition]{Question}
\newtheorem{questions}[definition]{Questions}

\newtheorem{example}[definition]{Example}
\newtheorem{problem}[definition]{Problem}
\newtheorem{exercise}[definition]{Exercise}

\numberwithin{theorem}{section}
\numberwithin{definition}{section}
\numberwithin{equation}{section}

\def\reals{{\mathbb R}}
\def\torus{{\mathbb T}}
\def\integers{{\mathbb Z}}
\def\rationals{{\mathbb Q}}
\def\naturals{{\mathbb N}}
\def\complex{{\mathbb C}\/}
\def\heis{{\mathbb H}\/}
\def\distance{\operatorname{distance}\,}
\def\diststar{{\rm dist}^\star}
\def\sym{\operatorname{Symm}\,}
\def\support{\operatorname{support}\,}
\def\dist{\operatorname{dist}}
\def\Span{\operatorname{span}\,}
\def\degree{\operatorname{degree}\,}
\def\kernel{\operatorname{kernel}\,}
\def\dim{\operatorname{dim}\,}
\def\codim{\operatorname{codim}}
\def\trace{\operatorname{trace\,}}
\def\Span{\operatorname{span}\,}
\def\dimension{\operatorname{dimension}\,}
\def\codimension{\operatorname{codimension}\,}
\def\Gl{\operatorname{Gl}}
\def\Sl{\operatorname{Sl}}
\def\nullspace{\scriptk}
\def\kernel{\operatorname{Ker}}
\def\ZZ{ {\mathbb Z} }
\def\p{\partial}
\def\rp{{ ^{-1} }}
\def\Re{\operatorname{Re} }
\def\Im{\operatorname{Im} }
\def\ov{\overline}
\def\eps{\varepsilon}
\def\lt{L^2}
\def\diver{\operatorname{div}}
\def\curl{\operatorname{curl}}
\def\etta{\eta}
\newcommand{\norm}[1]{ \|  #1 \|}
\def\expect{\mathbb E}
\def\bull{$\bullet$\ }
\def\det{\operatorname{det}}
\def\Det{\operatorname{Det}}
\def\multiR{\mathbf R}
\def\bestA{\mathbf A}
\def\bestB{\mathbf B}
\def\bestC{\mathbf C}
\def\Apq{\mathbf A_{p,q}}
\def\Apqr{\mathbf A_{p,q,r}}
\def\rank{\operatorname{rank}}
\def\rankk{\mathbf r}
\def\diameter{\operatorname{diameter}}
\def\bp{\mathbf p}
\def\bq{\mathbf q}
\def\bff{\mathbf f}
\def\bg{\mathbf g}
\def\bh{\mathbf h}
\def\bv{\mathbf v}
\def\bw{\mathbf w}
\def\bu{\mathbf u}
\def\bx{\mathbf x}
\def\bc{\mathbf c}
\def\bt{\mathbf t}
\def\by{\mathbf y}
\def\bz{\mathbf z}
\def\bzero{\mathbf 0}
\def\bL{\mathbf L}
\def\essd{\operatorname{essential\ diameter}}
\def\kbe{{\scriptk}_\be}

\def\mab{M}
\def\t2{\tfrac12}

\newcommand{\abr}[1]{ \langle  #1 \rangle}
\def\unitQ{{\mathbf Q}}
\def\mbfp{{\mathbf P}}

\def\aff{\operatorname{Aff}}
\def\T{{\mathcal T}}

\def\repair{\medskip \hrule \begin{center}{Construction Zone}\end{center}  \hrule \medskip}
\def\endrepair{\medskip \hrule\hrule\medskip}

\def\ovl{\overline{L}}
\def\bard{\bar\delta}
\def\tdelt{\tilde\delta}
\def\essinf{\operatorname{essinf}}
\def\esssup{\operatorname{esssup}}

\newcommand{\Norm}[1]{ \Big\|  #1 \Big\| }
\newcommand{\set}[1]{ \left\{ #1 \right\} }
\newcommand{\sset}[1]{ \{ #1 \} }

\def\one{{\mathbf 1}}
\newcommand{\modulo}[2]{[#1]_{#2}}

\def\Abest{{\mathbb A}}

\def\symdif{\,\Delta\,}
\def\defb{|E\symdif \bb|}
\def\bb{\mathbb B}

\def\akd{{\mathbf A}_{k,d}}
\def\ak{{\mathbf A}_{k}}

\def\scriptf{{\mathcal F}}
\def\scripts{{\mathcal S}}
\def\scriptq{{\mathcal Q}}
\def\scriptg{{\mathcal G}}
\def\scriptm{{\mathcal M}}
\def\scriptb{{\mathcal B}}
\def\scriptc{{\mathcal C}}
\def\scripth{{\mathcal H}}
\def\scriptt{\Phi}
\def\scripti{{\mathcal I}}
\def\scripte{{\mathcal E}}
\def\scriptv{{\mathcal V}}
\def\scriptw{{\mathcal W}}
\def\scriptu{{\mathcal U}}
\def\scripta{{\mathcal A}}
\def\scriptr{{\mathcal R}}
\def\scripth{{\mathcal H}}
\def\scriptd{{\mathcal D}}
\def\scriptl{{\mathcal L}}
\def\scriptn{{\mathcal N}}
\def\scriptp{{\mathcal P}}
\def\scriptk{{\mathcal K}}
\def\scriptP{{\mathcal P}}
\def\scriptj{{\mathcal J}}
\def\scriptz{{\mathcal Z}}
\def\frakv{{\mathfrak V}}
\def\frakG{{\mathfrak G}}
\def\frakA{{\mathfrak A}}
\def\frakB{{\mathfrak B}}
\def\frakC{{\mathfrak C}}
\def\frakf{{\mathfrak F}}
\def\fraki{{\mathfrak I}}
\def\fcross{{\mathfrak F^{\times}}}

\def\boldf{\mathbf f}
\def\bolda{\mathbf a}
\def\boldb{\mathbf b}
\def\boldg{\mathbf g}
\def\bF{\mathbf F}
\def\bG{\mathbf G}
\def\bE{\mathbf E}
\def\bI{\mathbf I}
\def\be{\mathbf e}
\def\br{\mathbf r}
\def\bA{\mathbf A}
\def\ba{\mathbf a}
\def\bB{\mathbf b}
\def\bs{\mathbf s}
\def\Estar{E^\star}
\def\bEstar{\mathbf E^\star}
\def\bEdagger{\mathbf E^\dagger}
\def\bAstar{\mathbf A^\star}
\def\Ee{{\mathcal E}}
\def\dorbit{\dist(\bE,\scripto(\bEstar))}
\def\scripto{{\operatorname{orbit}}}

\def\tf{\tfrac{1}{2}}
\def\rplus{{\reals^+}}
\def\bart{\bar t}
\def\fg{\frakG}

\author{Michael Christ}
\address{
        Michael Christ\\
        Department of Mathematics\\
        University of California \\
        Berkeley, CA 94720-3840, USA}
\email{mchrist@berkeley.edu}
\thanks{Research of both authors supported in part by NSF grant DMS-1363324.}

\author{Kevin O'Neill}
\address{
	Kevin O'Neill\\
        Department of Mathematics\\
        University of California \\
        Berkeley, CA 94720-3840, USA}
\email{kevinwoneill@berkeley.edu}

\date{November 30, 2017}

\title[Maximizers of Rogers-Brascamp-Lieb-Luttinger Functionals] 
{Maximizers \\ of 
Rogers-Brascamp-Lieb-Luttinger Functionals \\ in Higher Dimensions}

\begin{abstract} 
A symmetrization inequality of Rogers and of Brascamp-Lieb-Luttinger
states that for a certain class of multilinear integral expressions,
among  tuples of sets of prescribed Lebesgue measures,
tuples of balls centered at the origin are among the maximizers.
Under natural hypotheses,
we characterize all maximizing tuples for these inequalities
for dimensions strictly greater than $1$.
We establish a sharpened form of the inequality.
\end{abstract}

\maketitle

\section{Introduction}
Let $J$ be a finite index set, and  for each $j\in J$ let
$L_j:\reals^D\to\reals^d$ be a surjective linear mapping.
Let $E_j\subset\reals^d$ be measurable sets with finite Lebesgue measures,
let $\bE=(E_j: j\in J)$,
and let $f_j=\one_{E_j}$  denote the indicator function of $E_j$.
The objects of our investigation are functionals
\begin{equation*}
\Phi_\scriptl(\bE) = \int_{\reals^{D}} \prod_{j\in J} \one_{E_j}\circ L_j,
\end{equation*}
integration being with respect to Lebesgue measure.
The goal of this paper is to determine those tuples $\bE$ of sets,
among all tuples with specified Lebesgue measures,
that maximize $\Phi_\scriptl$ in the case $d>1$.

A structural hypothesis will be in force throughout this paper.
Let $\scriptl^1=\{L_j^1: j\in J\}$ be a finite family of 
surjective linear mappings $L_j^1:\reals^m\to\reals^1$.  
Let $m\ge 2$.
To any dimension $d\ge 2$ associate $D=md$ and $\scriptl=\scriptl^d =\{L_j^d: j\in J\}$
where $L_j^d: \reals^{md}\to\reals^d$
is defined as follows. Express $L_j^1(x_1,\dots,x_m)
= \sum_{i=1}^m a_{i,j}x_i$ with each $a_{i,j}\in\reals$.
For $\bx = (x_1,\dots,x_m)\in \reals^{md}$, these same coefficients
are used to define $L_j^d$, by
\begin{equation} \label{eq:structure}
L_j^d(\bx) = \sum_{i=1}^m a_{i,j}x_i.  \end{equation}
Each $L_j^d$ is then surjective. 
The general linear group $\Gl(d)$ acts diagonally on $\reals^{md} = (\reals^d)^m$ by
$\bx = (x_1,\dots,x_m)\mapsto A(\bx) = (Ax_1,\dots,Ax_m)$
for $A\in\Gl(d)$ and $\bx\in\reals^{md}$.
The structural hypothesis \eqref{eq:structure} is equivalent to
\begin{equation} \label{association}
L_j^d\circ A = A\circ L_j^d \ \text{ for every $j\in J$.} \end{equation}

For $d=1$, tuples $\bE$ that maximize $\Phi_\scriptl$,
among all tuples of indicator functions of sets of specified Lebesgue measures, 
have been characterized in \cite{christBLL}, under certain hypotheses.
In the present paper, we extend this result to higher dimensions, under the
structural hypothesis \eqref{eq:structure}, along with other natural hypotheses.

For any Lebesgue measurable set $E\subset\reals^d$ satisfying
$0<|E|<\infty$, define $E^\star\subset\reals^d$ to be the closed ball
centered at $0$ satisfying $|E^\star|=|E|$. 
Write $\bE = (E_j: j\in J)$
and $\bE^\star = (E_j^\star: j\in J)$.

The symmetrization inequality of
Rogers \cite{rogers2}\footnote{Inequality \eqref{eq:BLL} is widely attributed to the 1974 paper \cite{BLL}
of Brascamp, Lieb, and Luttinger. It was also published by Rogers \cite{rogers2} in 1957,
but the earlier publication has been largely overlooked. Moreover,  
\cite{rogers2} does not contain a rigorous 
identification of the limit of a sequence of Steiner symmetrizations.
On this point, \cite{rogers2} refers to \cite{rogers1}, which in turn refers to the book
Blaschke \cite{blaschke}. We were unable to find the justification in \cite{blaschke},
which is concerned with convex bodies rather than with general measurable sets.} 
and of Brascamp, Lieb, and Luttinger \cite{BLL} states that
\begin{equation}\label{eq:BLL} \scriptt_\scriptl(\bE) \le \scriptt_\scriptl(\bE^\star),  \end{equation}
provided that the mappings $L_j$ satisfy the symmetry hypothesis \eqref{eq:structure}.
Thus among tuples of sets with prescribed measures, the configuration in which each set
is a ball centered at the origin is a maximizer of $\scriptt_\scriptl$. 
Early contributions to this topic include works of
Riesz \cite{riesz}, Sobolev \cite{sobolev}, and Hardy, Littlewood, and P\'olya \cite{hardyetal}.
The literature concerning the Brunn-Minkowski inequality is also relevant.
The present paper is concerned with the inverse question of characterizing all maximizing tuples.

Other maximizers are generated from $\bEstar$ by symmetries.
Firstly, there is a natural action of the translation group $\reals^{md}$:
For $\by\in\reals^{md}$ define
$\tau_\by(\bE) = (E_j + L_j(\by): 1\le j\le m)$.
Then $\scriptt_\scriptl(\tau_\by(\bE)) =  \scriptt_\scriptl(\bE)$; 
in particular, $\bE$ is a maximizer if and only if $\tau_\by(\bE)$ is a maximizer.
Secondly, there is a natural action of
$\Sl(d)$, the subgroup of all elements of $Gl(d)$ with determinant $1$.
For any $A\in \Sl(d)$,  define
$A(\bE) = (A(E_j): j\in J)$.
Then $\scriptt_\scriptl(A(\bE)) = \scriptt_\scriptl(\bE)$
for all $m$--tuples $\bE$.
Thus among the maximizing configurations $\bE$ are all tuples of homothetic, compatibly
centered ellipsoids
\begin{equation} (A(B_j)+L_j(\by): j\in J), \end{equation}
where $B_j\subset\reals^d$ is the closed ball centered at the origin
of Lebesgue measure $|B_j|=|E_j|$, $\by\in\reals^{md}$, and $A\in \Sl(d)$.

Our main result states that under certain natural hypotheses on $\scriptl$
and on $\be=(|E_j|: j\in J)$, these are the only maximizers.

Accurate formulation of our results requires several definitions. 
In order to come to the point with reasonable promptitude, we state
our main result here in preliminary form, deferring those definitions to \S\ref{section:defns}.
The theorem is upgraded below to a more quantitative form, as Theorem~\ref{thm:stability}. 

\begin{theorem}[Uniqueness of maximizers, up to symmetry] \label{thm:main}
Let $d,m\ge 2$.  Let $J$ be a finite index set.  Let $\scriptl$
be a nondegenerate collection of 
linear mappings $L_j:\reals^{md}\to\reals^d$ satisfying the structural hypothesis \eqref{eq:structure}.
Let $\be\in(0,\infty)^J$. Suppose that $(\scriptl,\be)$ is strictly admissible.
Let $\bE$ be a $J$-tuple of Lebesgue measurable subsets of $\reals^d$
satisfying $|E_j|=e_j$ for each $j\in J$. Then
$\scriptt_{\scriptl}(\bE)=\scriptt_{\scriptl}(\bE^\star)$
if and only if there exist $\bv\in\reals^m$ and $\psi\in\Sl(d)$ satisfying
\begin{equation} E_j = \psi(E_j^\star)+L_j(\bv) 
\ \text{ for every $j \in J$.} \end{equation}
\end{theorem}

For the most fundamental example, the Riesz-Sobolev inequality, and more generally
when $|J|=m+1$, this was proved by Burchard \cite{burchard}. 
The case $(m,d)=(2,1)$, for $|J|$ arbitrarily large, was treated in \cite{christflock},
under a weaker admissibility and slightly stronger nondegeneracy hypotheses.
The nearly general\footnote{\cite{christBLL} assumes an auxiliary genericity condition.} 
case with $d=1$ was treated in \cite{christBLL}.
Related results, in which indicator functions of sets are replaced by more general functions,
at least one of which satisfies an auxiliary hypothesis of strict monotonicity of $f_j^\star$,
are discussed in \cite{liebloss} and references cited there.

It remains an open problem to determine all maximizers in the weakly admissible
(that is, admissible but not strictly admissible)
case for general $(m,d)$.
Theorem~\ref{thm:main} does not extend verbatim to the general weakly admissible case when $d\ge 2$. 
Indeed, for the Riesz-Sobolev inequality in dimensions $d\ge 2$,
Burchard \cite{burchard} characterizes maximizing tuples in the weakly admissible
case as arbitrary ordered triples of homothetic, compatibly translated convex sets. 

While Theorem~\ref{thm:main} may appear to be rather closely related to
stability for the Brunn-Minkowski inequality, analyzed in quantitative form
by Figalli and Jerison \cite{FJ1},\cite{FJ2},\cite{FJ3} after qualitative results by the first author
\cite{christbmtwo},\cite{christbmhigh}, the analysis here does not rely upon
machinery from discrete additive combinatorics.
As in the seminal work of Bianchi and Egnell \cite{bianchiegnell} on
quantiative stability for the Sobolev inequality,
we first reduce matters to configurations close to maximizers,
then exploit a perturbation expansion and spectral analysis
of certain compact linear operators. The reduction is based on monotonicity
of the functional under a continuous flow, as for instance in work of
Carlen and Figalli \cite{carlenfigalli} and Carlen \cite{carlen2017},
rather than on a compactness theorem as in \cite{bianchiegnell}.
The analysis of the eigenvalues, following \cite{christRSult}
and exploiting \cite{christBLL}, is indirect. 

A recent related work in a somewhat different spirit is Carlen and Maggi \cite{carlenmaggi},
in which a quantitative form of a theorem of Lieb \cite{lieb1977}
is obtained for the Riesz-Sobolev inequality. In \cite{carlenmaggi} and \cite{lieb1977}
there is a supplementary hypothesis on the distribution function
of one of the three factors, which introduces a smidgen of locality in comparison
to the general situation.

\section{Definitions, hypotheses, and second main theorem} \label{section:defns}

Maximizers can only be characterized usefully under appropriate hypotheses.
These are of two principal types: nondegeneracy hypotheses on $\scriptl$,
and admissibility hypotheses on the measures of the sets $E_j$. The formulation
of admissibility also involves $\scriptl$.

Let $m,d\ge 2$ be elements of $\naturals$.
Throughout this paper, $J$ will denote a finite index set.
Its cardinality is not determined by the parameters $m,d$. 

All scalar-valued functions in this paper are assumed to be real-valued.
For $d,m\in\naturals$ we will work with
$(\reals^d)^m$, and will identify this product with $\reals^{md}$.
Throughout this paper, two Lebesgue measurable subsets of Euclidean space 
are considered to be equal if their symmetric difference is a Lebesgue null set.
$B_k$ denotes the closed ball centered at the origin of Lebesgue measure $|B_k|=|E_k|$.
$\omega_d$ denotes the Lebesgue measure of the unit ball $\{x\in\reals^d: |x|\le 1\}$.

The following notational convention will be employed. 
We will work systematically with objects associated to $\reals^1$,
which in a natural way give rise to closely related objects associated to $\reals^d$ for each $d>1$.
We use the superscripts $1,d$ to indicate this association.
Specializing $d$ to $1$ will recover the object associated to $\reals^1$.
The family $\scriptl^1$ is one such object. 
Associated to it is
$\scriptl^d=(L_j^d: j\in J)$, a family of linear mappings $L_j^d: \reals^{md}\to\reals^d$,
defined in terms of $\scriptl^1$  by \eqref{eq:structure}.
We will often simplify notation by suppressing the superscripts,
writing simply $\scriptl,L_j$ for $\scriptl^d,L_j^d$, respectively.

\begin{definition}[Nondegeneracy] \label{defn:nondegenerate}
Let $J$ be a finite index set, and let $m\ge 2$. 
A family $\scriptl=\{L_j: j\in J\}$ of linear mappings
$L_j:\reals^m\to \reals$ is nondegenerate if
\newline
(i) Each $L_j$ is surjective, 
\newline (ii)
For any $i\ne j\in J$, $L_i$ is not a scalar multiple of $L_j$,
\newline (iii) For each $j\in J$, $\bigcap_{i\ne j\in J} \kernel(L_i)=\{0\}$.

For $d\ge 2$, the associated family $\scriptl^d$ is nondegenerate if $\scriptl^1=\scriptl$ is nondegenerate.
\end{definition}

Condition (iii) forces $|J|\ge m+1$. 

To any $\be=(e_j: j\in J)\in (0,\infty)^J$ we associate
$\br=(r_j: j\in J)$,  defined by $\omega_d r_j^d=e_j$.
If $|E_j|=e_j$ then $E_j^\star$ is the closed ball centered at the 
origin, of radius $r_j$.

\begin{notation}[Associated convex bodies]
Let $\be\in(0,\infty)^J$ and $\scriptl=(L_j: j\in J)$.
$\scriptk^1_\be\subset\reals^m$ is the closed convex set 
\begin{equation}
\scriptk^1_\be = \set{x\in\reals^m: |L_j^1(x)|\le\tfrac12 e_j\ \text{ for each $j\in J$} }.
\end{equation} 

$\scriptk^d_\be\subset\reals^{md}$ is the closed convex set 
\begin{equation}
\scriptk^d_\be = \set{\bx\in\reals^{md}: |L_j^d(\bx)|\le r_j\ \text{ for each $j\in J$} }.
\end{equation}
\end{notation}

We will sometimes suppress the superscript, writing simply $\scriptk_\be$.
$\scriptk_\be$ contains an open neighborhood of $0$.
Condition (iii) of the nondegeneracy hypothesis ensures that $\scriptk_\be$ is compact. 

\begin{notation}[Kernels $K_j$] \label{defn:K_j}
For $j\in J$, $K_j = K_{j,\be,\scriptl}^d$ is the radially symmetric function from
$\reals^d$ to $[0,\infty]$ defined  by
\begin{equation} \label{eq:Kjdefn}
\int_{A} K_j = \scriptt_\scriptl(\bE) \end{equation}
for every Lebesgue measurable set $A\subset\reals^d$,
where $E_j=A$ and for every $i\in J\setminus\{j\}$, 
$E_i$ is the closed ball centered at $0\in\reals^d$ of Lebesgue measure $e_i$.
\end{notation}
The parameter $e_j$ does not enter into the definition of $K_j$.

Under the nondegeneracy hypothesis, $K_j$ is finite-valued and continuous. 
$K_j$ is radially symmetric, 
and $[0,\infty)\owns r \mapsto K_j(rx)$ is nonincreasing for each $x\in\reals^d$.
As $r$ varies with $x$ held fixed, the quantities $K_j(rx)$ represent the $(m-1)d$--dimensional
volumes of parallel slices of an $md$--dimensional convex body, so that
according to the Brunn-Minkowski inequality, $r\mapsto \log K_j(rx)$ is concave in 
the interval in which $K_j$ is strictly positive.

Therefore the one-sided derivatives 
$D^\pm K_j(x) = \lim_{h\to 0^\pm} h^{-1} (K_j(x+h)-K_j(x))$
exist and are finite and nonpositive whenever $x>0$ and $K_j(x)>0$.
If $D^+K_j(tx)<0$ then for any $t'>t$ for which $K_j(t'x)>0$,
$D^\pm K_j(t'x)$ are also negative, and $K(tx)>K(t'x)$.

We write $\be^{1/d}$ as shorthand for $(e_j^{1/d}: j\in J)$.
The following definitions of admissibility and strict admissibility for $d=1$
are taken from \cite{christBLL}.

\begin{definition}[Admissibility for $d=1$] \label{defn:admissible}
$(\scriptl^1,\be)$ is admissible if for each $k\in J$
there exists $\bx\in \scriptk_\be^1$ satisfying $|L_k^1(\bx)| = e_k/2$.
\end{definition}

\begin{definition}[Strict admissibility for $d=1$] \label{defn:strictlyadmissible}
Let $J$ be a finite index set.
Let $\scriptl^1=(L_j^1: j\in J)$ be a nondegenerate $J$-tuple of 
linear mappings $L_j^1:\reals^m\to\reals^1$.
Let $\be = (e_j: j\in J)\in(0,\infty)^J$.
Then $(\scriptl^1,\be)$ is strictly admissible  if 
the following two conditions hold for each $j\in J$.

(i) 
There exists $\bx\in\scriptk_\be$ satisfying $|L_j^1(\bx)| = \tfrac12 \be_j$
and $|L_i^1(\bx)|< \tfrac12\be_i$ for all $j\ne i\in J$.

(ii)
$D^- K_j^1(e_j/2)<0$. 
\end{definition}

When $|J|=3$ and $m=2$, if $\scriptl$ is nondegenerate
then it is possible to make linear changes of variables
after which $\scriptl^1$ consists of the three mappings $(x_1,x_2)\mapsto x_1,x_2,x_1+x_2$;
see the discussion of symmetries following Notation~\ref{Mijdefn}.
Strict admissibility is then equivalent to strict admissibility as defined
by Burchard \cite{burchard} in the context of the Riesz-Sobolev inequality.
The relevance of an admissibility hypothesis for a conclusion
of the type of Theorem~\ref{thm:main} 
to the Riesz-Sobolev inequality is explained in \cite{burchard}, and the same
explanation applies in the general context.

\begin{definition}[Strict admissibility for $d>1$] \label{defn:strictlyadmissiblehigher}
Let $\scriptl^1=(L_j^1: j\in J)$ be a nondegenerate $J$-tuple of 
linear mappings $L_j^1:\reals^m\to\reals^1$.
Let $\scriptl^d$ be associated to $\scriptl^1$ by \eqref{eq:structure}.
Let $\be = (e_j: j\in J)\in(0,\infty)^J$.
Then $(\scriptl^d,\be)$ is admissible if $(\scriptl^1,\be^{1/d})$ is admissible. 
$(\scriptl^d,\be)$ is strictly admissible  if 
$(\scriptl^1,\be^{1/d})$ is strictly admissible.  
\end{definition}

$(\scriptl^d,\be)$ is said to be weakly admissible if it is admissible, but not strictly admissible.

When $\scriptl$ is nondegenerate and $(\scriptl,\be)$ is strictly admissible,
$\scriptk_\be^1$ is a compact convex polytope in $\reals^m$ that contains
a neighborhood of the origin. Therefore if it is nonempty
then it has finitely many extreme points, and is equal to the convex hull of those points.
For each extreme point $\bx$, there must exist
at least $m$ indices $k\in J$ for which $|L_k^1(\bx)| = e_k/2$.
Moreover, $\{L_j^1\in\scriptl^1: |L_j^1(\bx)|=e_j/2\}$ 
must span the dual space of $\reals^m$.

Let $\scripto(\bEstar)$ denote the orbit of $\bEstar$
under the group generated by the translation symmetry group $\reals^m$ 
and by $\Sl(d)$, acting diagonally.

\begin{definition}
The distance from $\bE$ to the orbit of $\bE^\star$ is
\begin{equation} \label{eq:distancedefn}
\dist(\bE,\scripto(\bE^\star)) = \inf_{\bv,\psi} 
\max_{j\in J} |E_j\symdif (\psi(E_j^\star) + L_j(\bv))|
\end{equation}
where the infimum is taken over all $\bv\in\reals^{md}$
and all $\psi\in\Sl(d)$.
\end{definition}

It is elementary that for each tuple $\bE$ of sets with finite, positive measures, 
the infimum in this definition is attained by some $\bv,\psi$.

We also establish a quantitative form of Theorem~\ref{thm:main}.

\begin{theorem}[Quantitative stability] \label{thm:stability}
Let $d,m,J,\scriptl$ be as in Theorem~\ref{thm:main}.
let $S$ be a compact subset of the Cartesian product of $(0,\infty)^J$
with the set of all $J$--tuples of surjective linear mappings
such that $\scriptl$ is nondegenerate and $(\scriptl,\be)$ is strictly admissible
for every $(\scriptl,\be)\in S$. 
Then there exists $c>0$ such that for every $(\scriptl,\be)\in S$, and for every
$J$-tuple $\bE$ of Lebesgue measurable subsets of $\reals^d$
satisfying $|E_j|=e_j$ for each $j\in J$,
\begin{equation} \label{eq:mainconclusion}
\scriptt_{\scriptl}(\bE)\le \scriptt_{\scriptl}(\bE^\star) - c\dist(\bE,\scripto(\bE^\star))^2.
\end{equation} 
\end{theorem}

The exponent $2$ in \eqref{eq:mainconclusion} is optimal.

Stability theorems have often been established after maximizers have been characterized. 
As was done in \cite{christBLL}, we will instead prove the quantitative stability theorem directly,
obtaining the characterization of maximizers as a corollary. 
We are not in possession of a direct proof of 
Theorem~\ref{thm:main} that does not go through \eqref{eq:mainconclusion}.

The proof of Theorem~\ref{thm:stability} combines analyses developed 
in \cite{christRSult} (for stability of the Riesz-Sobolev inequality in
dimensions greater than $1$) and in \cite{christBLL} (for the $d=1$ 
analogues of Theorems~\ref{thm:main} and \ref{thm:stability}), respectively.
The present paper also invokes the case in which $d=1$ and each set $E_j$
is an interval, whose treatment is one of the main steps in \cite{christBLL}.
Several auxiliary results from those works are also invoked here.

The analogues of Theorems~\ref{thm:main} and \ref{thm:stability} established for $d=1$ in \cite{christBLL}
included a supplementary hypothesis of genericity.
This hypothesis is reviewed in the proof of Lemma~\ref{lemma:Es} below.
It is interesting that the method of analysis requires no genericity hypothesis in higher dimensions.
The Gowers multilinear forms, in the situation in which
all sets have equal Lebesgue measures, provide natural examples
in which genericity fails to hold.
\cite{christgowersnorms} treats these forms, by a different and more specialized argument.

Various differences between the cases $d=1$ and $d>1$ arise in the proofs.

\section{Two families of kernels}

Throughout the remainder of the paper, $\scriptl$ is assumed to be
nondegenerate, and $(\scriptl,\be)$ to be strictly admissible, even where
these hypotheses are not stated explicitly.

Kernels $K_j$ have been introduced above. These will appear in first-order terms in a perturbation
expansion of $\scriptt(\bE)$ about $\bEstar$. In this section we introduce corresponding kernels $M_{i,j}$ that
arise in second-order terms of the same expansion, and derive needed properties of both
$K_j$ and $M_{i,j}$.

For the remainder of the paper it is assumed that $d\ge 2$ and $m\ge 2$.
We employ the following notational convention.
Superscripts $d$ will be exhibited on quantities such as $\scriptl^d,L_j^d,K_j^d$
only when deemed necessary to avoid confusion. If no superscript is written,
then $\scriptl,L_j,K_j$ are to be understood to indicate quantities
$\scriptl^d,L_j^d,K_j^d$
associated to
$\scriptl^1,L_j^1,K_j^1$
in the manner prescribed above.
Statements concerning these are to hold for arbitrary $d\ge 1$,
or for arbitrary $d\ge 2$, as indicated.
This convention applies equally to the kernels $M_{i,j}^d$ introduced next.
In the same spirit, $\Phi=\Phi_\scriptl$. $B_j= E_j^\star$ is the closed
ball centered at $0\in\reals^d$ of radius $r_j$.

\begin{notation}[Kernels $M_{i,j}$] \label{Mijdefn}
For each pair $i\neq j\in J$, $M_{i,j}=M_{i,j}^d:\reals^d\times\reals^d\rightarrow[0,\infty)$ 
is defined through the relation
\begin{equation}\label{eq:Lij}
\iint M_{i,j}(x,y)f_i(x)f_j(y)dxdy=\scriptt_{\scriptl}(\bg),
\end{equation}
where $\bg=(g_n:n\in J)$ defined by $g_i=f_i, g_j=f_j$, and $g_k=\one_{B_k}$ for every $k\notin\{i,j\}$. 
\end{notation}

Two symmetries, besides the translation action of $\reals^{md}$ and the diagonal
action \eqref{association} of $\Gl(d)$, will be useful in the analysis.
The first of these symmetries is an action of
the product group $(\rplus)^J$
on tuples $\scriptl =(L_j: j\in J)$. 
Write $\br=(r_j: j\in J)\in (\rplus)^J)$.
The action is $\br(\scriptl) = (r_j L_j: j\in J)$,
where $r_j L_j(x)$ is the product of the scalar $r_j$ with $L_j(x)$.
Tuples of sets $\bE$ and of Lebesgue measures $\be$ transform in corresponding ways.
Thus
\begin{equation} \label{dilationaction} \scriptt_{\br(\scriptl)}(\bE) 
= \scriptt_{\scriptl}(\br(\bE))
\end{equation}
where $\br(\bE)=(r_j^{-1} E_j: j\in J)$ for $\bE=(E_j: j\in J)$.

A second symmetry is $(L_j: j\in J)\mapsto (L_j\circ \widehat{A}: j\in J)$
for arbitrary $A\in\Gl(m)$. $\widehat{A}$ denotes here the
diagonal action of $A$ on $(\reals^d)^m = \reals^{md}=(\reals^m)^d$,
by $(x_k: 1\le k\le d)\mapsto (Ax_k: 1\le k\le d)$, with each $x_k\in\reals^m$.

\begin{lemma} \label{lemma:smallmeasure}
Let $i\ne j\in J$. For $\delta>0$ define $\Omega_{i,j}(\delta)$ to be
the set of all $(\theta_i,\theta_j)\in S^{d-1}\times S^{d-1}$
for which there exists $k\in J\setminus\{i,j\}$ such that
$\{L_i,L_j,L_k\}$ is linearly dependent and 
$|L_k(r_i\theta_i,r_j\theta_j)|\le C\delta$.
There exists $C'<\infty$ such that for all
sufficiently small $\delta>0$, 
\begin{equation}(\sigma\times\sigma)(\Omega_{i,j}(\delta))\le C'\delta.\end{equation}
\end{lemma}

\begin{proof}
Write $\bx = (x_1,x_2,x')\in \reals^d\times\reals^d\times (\reals^d)^{m-2}$.
Since $L_i^1$ and $L_j^1$ are linearly independent,
it is possible to make a linear change of variables,
after which $L_i(x_1,x_2,x')=x_1$  and $L_j(x_1,x_2,x')=x_2$.

Consider any $k\in J\setminus\{i,j\}$ such that $\{L_i,L_j,L_k\}$ is linearly dependent;
if there is no such index $k$, then the conclusion holds vacuously.
It suffices to bound the measure of the set of all $(\theta_i,\theta_j)$ that satisfy 
$|L_k(r_i\theta_i,r_j\theta_j)|\le \delta$ for this particular index,
for there are only finitely many such indices. 
$L_k(\bx)$ can be expressed as $L_k(x_1,x_2,x')=a x_1 + b x_2$ for some coefficients $a,b\in\reals$.
Neither coefficient can vanish, since $L_k$ is a scalar multiple of neither $L_i$ nor $L_j$.

Then
\begin{align*}|a r_i\theta_i+b r_j\theta_j|^2-r_k^2
 = (a^2 r_i^2 + b^2 r_j^2 -r_k^2) + 2a b r_ir_j \theta_i\cdot\theta_j
= 2a br_ir_j \big( s +  \theta_i\cdot\theta_j\big)
\end{align*}
where
$s=\frac{a^2 r_i^2 + b^2 r_j^2 -r_k^2}{2a b r_ir_j}$.
The coefficient $2a b r_i r_j$ is nonzero; its precise value is of no importance here.
We claim that the strict admissibility hypothesis guarantees that $|s|<1$.
Granting this,
for any $\theta_i\in S^{d-1}$ and any $s\in(-1,1)$,
the set of all $\theta_j\in S^{d-1}$ satisfying $|s + \theta_i\cdot\theta_j|\le \eta$
has measure $O(\eta)$ with respect to $\theta$.
The desired conclusion then follows immediately.

We claim that 
$r_k < |a|r_i+|b|r_j$. The proof is by contradiction.
If $r_k\ge |a|r_i+|b|r_j$
then any $\bx\in\reals^2$ satisfying $|L_n^1(\bx)|< r_n$ for 
every $n\in J\setminus\{k\}$ would in particular satisfy these inequalities
for $n=i,j$, and therefore $|L_k^1(\bx)| \le |a|\cdot|x_1| + |b|\cdot|x_2|
< |a|r_i+|b|r_j\le r_k$. This contradicts the strict admissibility hypothesis,
which states that there exists $\bx$ satisfying $|L_k(\bx)|=r_k$
but $|L_n(\bx)|<r_n$ for every $J\owns n\ne k$.

Likewise, $\big|\,|a|r_i-|b|r_j\,\big| < r_k$. 
If not, suppose without loss of generality that $|a|r_i\ge |b|r_j$.
Then $|a|r_i \ge |b|r_j + r_k$. 
If $|L_n^1(\bx)|< r_n$ for every $n\in J\setminus\{i\}$
then 
\[ |L_i^1(\bx)| = |x_1| = \big| a^{-1}(ax_1+bx_2) - a^{-1}bx_2 \big|
< |a|^{-1} r_k + |a|^{-1}|b|r_j \le r_i,
\]
which again is a contradiction.

Now squaring both inequalities, and combining the results, yields the claim.
\end{proof}

The following elementary inequality will be invoked several times in the sequel:
For any surjective linear mappings $\ell_j:\reals^{m'd}\to\reals^d$,
any $v_i\in\reals^d$,
and any nonnegative Lebesgue measurable functions $f_j$,
\begin{equation} \label{trivialHBL}
\int_{\reals^{m'd}} \prod_{i\in I} f_i(\ell_i(\bx)+v_i)\,d\bx
\le C\prod_{i\in I'} \norm{f_i}_1\prod_{i\in I''}\norm{f_i}_\infty
\end{equation}
where
$I$ is the disjoint union of $I',I''$,
$|I'|=m'$,
$C<\infty$ depends only on $\{\ell_i: i\in I'\}$, 
and  $\bigcap_{i\in I'} \kernel(\ell_i)=\{0\}$.
This can be proved by 
majorizing the integrand by the product of
$\prod_{i\in I''}\norm{f_i}_\infty$
with
$\prod_{i\in I'} f_i(\ell_i(\bx)+v_i)$,
factoring out the $L^\infty$ norms,
identifying $\reals^{m'd}$ with $\prod_{i\in I'} \reals^d$
via an invertible linear transformation,
and making  an affine change of variables, after which $\bx\mapsto \ell_i(\bx)+v_i$
is the projection onto the $i$--th factor.

\begin{lemma}\label{lemma:Holder}
Let $m\ge 3$. 
Let $k\ne k'\in J$. Suppose that for each $n\ne k,k'$,
$\{\scriptl_k^1,\scriptl_{k'}^1,\scriptl_n^1\}$ is linearly independent. Then
$M_{k,k'}$ is Lipschitz continuous. 
\end{lemma}

\begin{proof}
After a change of variables making $L_k(x)\equiv x_1$ and $L_{k'}(x)\equiv x_2$, we have
\begin{equation}
M_{k,k'}(w,z) = 
c_{k,k'}\int_{\reals^{(m-2)d}} \prod_{j\ne k,k'} f_j(\tilde L_j(\by) + \psi_j(w,z))\,d\by
\end{equation}
where $f_j=\one_{B_j}$, $\by = (y_j: J\owns j\ne k,k')$, and $\tilde L_j,\psi_j$ are 
linear mappings to $\reals^d$ determined by $\{L_j\}$. 
The positive scalars $c_{k,k'}$ are reciprocals of absolutes value of Jacobian determinants. 

By multilinearity of this expression and
the triangle inequality, in order to control
$M_{k,k'}(w,z)-M_{k,k'}(w',z')$, it suffices to analyze
\begin{equation}
\int_{\reals^{(m-2)d}} \prod_{j\ne k,k'} g_j
\,d\by
\end{equation}
where $g_n = f_n(\tilde L_n(\by) + \psi_n(w,z)) - f_n(\tilde L_n(\by) + \psi_n(w',z')) $
for a single index $n\in J\setminus\{k,k'\}$,
and each index $j$ not equal to any of $k,k',n$,
the factor $g_j$ takes one of the two forms
$f_j(\tilde L_j(\by) + \psi_j(w,z))$
or
$f_j(\tilde L_j(\by) + \psi_j(w',z'))$.

This integral has absolute value majorized by
$C\norm{f_n(\cdot + \eta)-f_n(\cdot)}_{L^1(dy)}$
where $\eta = [\psi_n(w,z)-\psi_n(w',z')]$
and $C<\infty$ depends on $\scriptl$ and on $\be$.
Indeed, 
for $m=3$ this is immediate, since $\norm{g_j}_\infty\le 1$ for every $j\ne n$
and $\tilde L_n^d:\reals^d\to\reals^d$ is invertible by hypothesis.
If $m\ge 4$ then again by a hypothesis of the lemma,
$\{L_k,L_{k'},L_n\}$ is linearly independent and therefore
there exists a subset $J'\subset J\setminus\{k,k'\}$
of cardinality $m-2$
such that $n\in J'$ and $\cap_{i\in J'}\kernel(\tilde L_i)=\{0\}$.
Thus \eqref{trivialHBL} can be invoked.

Finally,
$\norm{f_n(\cdot + \eta)-f_n(\cdot)}_{L^1(dy)}$
is a Lipschitz continuous function of $\eta$,
and hence of $(w,z)-(w',z')$.
\end{proof}

The conclusion of Lemma~\ref{lemma:Holder} does not hold for $m=2$,
when $M_{i,j}$ is a constant multiple of a product of indicator functions.
The next result provides a weaker conclusion that will suffice for our purpose.

\begin{lemma}\label{lemma:Holder2}
Let $m,d\ge 2$. 
Let $i\ne j\in J$. 
Then $M_{i,j}$ can be expressed as the product of a Lipschitz continuous function with 
$\one_{\Omega(i,j)}$, where the $\sigma\times\sigma$
measure of the set of all $(\theta,\theta')\in S^{d-1}\times S^{d-1}$
such that $(r_i\theta,r_j\theta')$ lies within distance
$\delta$ of the boundary of $\Omega(i,j)$ is $O(\delta)$.
\end{lemma}

In this statement, $\Omega(i,j)$ could be $\reals^d\times\reals^d$,
so that $\one_{\Omega(i,j)}\equiv 1$ and hence $M_{i,j}$ is Lipschitz continuous;
but this is not the general case.

\begin{proof}
Let $J''$ be the set of all $k\in J\setminus\{i,j\}$ for which $L_k$
belongs to the span of $L_i,L_j$. Let $J'=J\setminus (J''\cup\{i,j\})$.
Then after a change of variables, $L_i(x)\equiv x_1$, $L_j(x)\equiv x_2$,
and with the same notations as in the proof of Lemma~\ref{lemma:Holder},
\begin{equation} \label{eq:forholder}
M_{i,j}(w,z) = c_{i,j} \prod_{k\in J''} \one_{T_{k}}(w,z)
\int_{\reals^{(m-2)d}} \prod_{k\in J'} \one_{B_k}(\tilde L_k(\by)+\psi_k(w,z))\,d\by,
\end{equation}
where $T_k$ is the set of all $(w,z)$ for which $L_k(\bx)\in B_k$
whenever $x_1=w$ and $x_2=z$. Recall that for each $k\in J''$, $L_k(\bx)$ depends only 
on the components $x_1,x_2$ of $\bx$. 
The integral in \eqref{eq:forholder} is to be interpreted as $1$ if $m-2=0$.

The integral in \eqref{eq:forholder} defines a Lipschitz continuous function of $(w,z)$, 
by the proof of Lemma~\ref{lemma:Holder}.
The set $T= \bigcap_{k\in J''}T_k$ has the required property,
by Lemma~\ref{lemma:smallmeasure}.
\end{proof}

\begin{lemma}\label{lemma:differentiability}
Let $m,d\ge 2$.
If $(\scriptl,\be)$ is strictly admissible, then
the left derivative of $t\mapsto K_i(t,0,\dots,0)$
is strictly negative at $t=r_i$.
\end{lemma}

Since the $K_i$ are radial, this is equivalent to strict negativity
of the corresponding one-sided radial directional derivative
of $K_i$ at each point $x\in\reals^d$ satisfying $|x|=r_i$.
For $d=1$, strict negativity of the left derivative
is part of the definition of strict admissibility.

\begin{proof}
Let $d\ge 2$.
Identify $\reals^d$ with $\reals^1\times\reals^{d-1}$, with coordinates
$z = (z_1,z')$ with $z'\in\reals^{d-1}$.
Thus $\reals^{(m-1)d}$
is identified with a product of $m-1$ factors of
$\reals^1\times\reals^{d-1}$,
with coordinates $\bz = ((u_1,w_1),\dots, (u_{m-1},w_{m-1}))$
where each $u_k\in\reals^1$ and $w\in\reals^{d-1}$.
Then 
\begin{equation} \label{Ki:iterated}
K_i(t,0,\dots,0)
= c_i \int_{\bw\in \reals^{(m-1)(d-1)}}
\int_{\bu\in\reals^{m-1}}
\prod_{j\ne i}
\one_{|L_j^d(\bw,\bu,(t,0,\dots,0))|\le r_j}\,d\bu\,d\bw
\end{equation}
for a certain positive constant $c_i$.
$L_j^d(\bw,\bu,(t,0,\dots,0))$
is naturally viewed as \[(L_j^1(\bw,t),L_j^{d-1}(\bu,0))\in \reals^1\times\reals^{d-1},\]
and thus
\begin{equation}
|L_j^d(\bw,\bu,(t,0,\dots,0))|^2
= |L_j^1(\bw,t)|^2 + |L_j^{d-1}(\bu,0)|^2.
\end{equation}
Thus the inner integral in \eqref{Ki:iterated} is rewritten as
\begin{equation} \label{Ki:inner}
K_{i,\bw}(t) =\int_{\bu\in\reals^{m-1}} \prod_{j\ne i}
\one_{|L_j^1(\bu,t)| \le (r_j^2-|L_j^{d-1}(\bw,0)|^2)^{1/2}} \,d\bu,
\end{equation}
and \eqref{Ki:iterated} expresses 
$K_i(t,0,\dots,0)$ as the integral with respect to $\bw$ of
one-dimensional kernels $K_{i,\bw}(t)$.
For each fixed $\bw$,
this integral is precisely the definition of the corresponding kernel $t\mapsto K_i^1(t)$
associated to $\scriptl^1$ in the theory for $d=1$.
with each $e_j$ replaced by $(r_j^2-|L_j^{d-1}(\bw,0)|^2)^{1/2}$.
Therefore this is a logarithmically concave function of $t$.

The second clause of our Definition~\ref{defn:strictlyadmissiblehigher} of strict admissibility
states that the left derivative of $t\mapsto K_i^1(t)$ is strictly negative at $t=r_i$.
This strict negativity is stable under small perturbations of $\bw$,
because $K_{i,\bw}^1(t)$ is continuous in $\bw$, strictly positive near $(\bw,0)$, 
and logarithmically concave with respect to $t$.
Logarithmic concavity, together with the property that $t\mapsto K_{i,\bw}^1(t)$
assumes its maximum value at $t=0$,
also ensure that this one-sided derivative 
exists and is nonpositive for every $\bw$.
Therefore the integral with respect to $\bw$ satisfies the stated conclusion.
\end{proof}

\begin{lemma}\label{lemma:diff2}
Suppose that $d,m\ge 2$.
Then for each $i\in J$, $K_i$ is continuously differentiable at $|y|=r_i$.
\end{lemma}

For $d=1$,
the conclusion also holds under a genericity hypothesis
imposed in \cite{christBLL}. 
It also holds for $m\ge 3$ under the additional hypothesis that any three elements of $\scriptl$ are
linearly independent.
These facts are not invoked in this paper, so their proofs are omitted.

\begin{proof}[Proof of Lemma~\ref{lemma:diff2}] 
Let $d>1$.
By exploiting the symmetry \eqref{dilationaction},
we may change variables by dilations in $\reals^1$
and by a linear transformation of $\reals^m$ to reduce to the situation in which 
$\reals^m$ is equipped with coordinates $(t,u,v)\in\reals^1\times\reals^1\times \reals^{m-2}$,
$L_i^1(t,u,v)\equiv t$,
and for each $j\ne i$, $L_j^1(t,u,v)= s_j t + u + \tilde L_j^1(v)$
where $\tilde L_j^1:\reals^{m-2}\to\reals$ is linear and surjective, each $s_j\in\reals$,
and the coefficients $s_j$ are pairwise distinct.
Indeed, because no $L_j$ is a scalar multiple of another,
the coefficients $s_j$ become pairwise distinct upon any generic rotation
of $\reals^{m-1}$.

$K_i$ can be expressed as
\begin{equation} 
K_i(t,0,\dots,0)
= 
c_i 
\int
\prod_{j\ne i}
\one_{|L_j^1(t,u,v)| \le (r_j^2-|L_j^{d-1}(0,\bw)|^2)^{1/2}} \,du\,dv
\,d\bw,
\end{equation}
where $c_i$ is a positive constant and the integral is over all
$(u,v,w)\in\reals\times\reals^{m-2}\times\reals^{(d-1)(m-1)}$.
The $j$-th factor in the integrand is the indicator function
of an interval \[[-s_j t-\tilde L_j(v)-\rho_j(v,\bw),\,\, -s_j t-\tilde L_j(v)+\rho_j(v,\bw)]\] 
evaluated at $u$, where 
\begin{equation} \rho_j(\bw) = (r_j^2-|L_j^{d-1}(0,\bw)|^2)^{1/2}
\ \text{ if $|L_j^{d-1}(0,\bw)|\le r_j$,} \end{equation}
and $\rho_j(\bw)=0$ otherwise. 

Therefore 
\begin{equation}  \label{Ki1}
K_i(t,0,\dots,0) = c_i 
\int_{v\in\reals^{m-2}}
\int_{\bw\in \reals^{(m-1)(d-1)}}
\lambda(v,\bw,t)
\,dv \,d\bw
\end{equation}
where
\begin{equation} 
\left\{ \begin{aligned}
a(v,\bw,t) &= \max_k [-s_k t-\tilde L_k(v) -\rho_k(\bw)],
\\
b(v,\bw,t) & =\min_j [-s_j t -\tilde L_j(v) +\rho_j(\bw)],  
\\
\lambda(v,\bw,t) &=\max(b(v,\bw,t)-a(v,\bw,t),\,0). 
\end{aligned} \right. \end{equation}

The function $t\mapsto\lambda(v,\bw,t)$ 
is Lipschitz continuous, uniformly in $v,\bw$.
Therefore
for any subset $S\subset\reals^{m-2}\times \reals^{(d-1)(m-1)}$ of finite Lebesgue measure,
\begin{equation} \label{SLip}
\big\|\int_S \lambda(v,\bw,t) \,dv \,d\bw\big\|_{\text{Lip}} \le C|S|,
\end{equation}
where $C<\infty$ depends only on $\scriptl,i$,
and $\norm{\cdot}_{\text{Lip}}$ denotes the Lipschitz norm of
a function whose domain is $\reals^1$.

Fix an arbitrary point $\bart\ne 0$, for which it is to be shown that $K_i$ is
continuously differentiable in a neighborhood of $\bart$.
Let $\Omega\subset\reals^{(m-1)(d-1)}$ be the set of all $\bw$ that satisfy $|L_j^{d-1}(0,\bw)|<r_j$
for every $j\in J\setminus\{i\}$. 
Thus $\lambda(v,\bw,t)\equiv 0$ for any $\bw\notin\Omega$, for all $v$.
$\Omega$ is bounded, open, and convex, hence connected. 

Consider the functions $-s_j \bart-\tilde L_j(v)+\rho_j(\bw)$ and $-s_k \bart -\tilde L_k(v) + \rho_k(\bw)$
for arbitrary distinct indices $j,k$, neither of which equals $i$.
The set of $(v,\bw)\in \reals^{m-2}\times \Omega$ for which these two quantities are equal
is an analytic variety of positive codimension. 
This holds because
$(s_j-s_k) \bart\ne 0$,
$\Omega$ is open and connected,
and the gradient of  $\bw\mapsto \rho_k(\bw)-\rho_j(\bw)$
is a real analytic vector-valued function on $\Omega$
that does not vanish identically.
This, in turn, follows from the definition
$\rho_l(\bw) = (r_l^2-|L_l^{d-1}(\bw)|^2)^{1/2}$, 
because $L_j^{d-1}$ and $L_k^{d-1}$ are linear mappings,
neither of which is a scalar multiple of the other. 
The same reasoning applies to any pair of functions
$-s_j \bart-\tilde L_j(v)-\rho_j(\bw)$ and $-s_k \bart -\tilde L_k(v) - \rho_k(\bw)$,
as well as to any pair
$-s_j \bart-\tilde L_j(v)-\rho_j(\bw)$ and $-s_k \bart -\tilde L_k(v) + \rho_k(\bw)$,
provided $j\ne k\in J\setminus\{i\}$.

Any real analytic variety of positive codimension has Lebesgue measure
equal to $0$. Moreover, 
the Lebesgue measure of the set of all 
points within any bounded region
that are within distance $\delta$
of such a variety, tends to $0$ as $\delta\to 0$.
Therefore for $t=\bart$, the region of integration in \eqref{Ki1}
can be replaced by 
the subset of $\reals^{m-2}\times \Omega$
in which no two of these functions are equal. 
Thus
\begin{equation} 
K_i(\bart,0,\dots,0)
= 
\sum_{j\ne k} 
\iint_{\Omega_{j,k}} (-s_j \bart-\tilde L_j(v) + \rho_j(\bw) +s_k\bart + \tilde L_k(v)-\rho_k(\bw)) \,dv \,d\bw
\end{equation}
where the sum is over all distinct $j,k\in J\setminus\{i\}$, 
the sets $\Omega_{j,k}$ are pairwise disjoint and open,
$b(v,\bw,\bart) = b_j(v,\bw,\bart)>a_k(v,\bw,\bart)=a(v,\bw,\bart)$ for all $\bw\in \Omega_{j,k}$,
$\Omega_{j,k}\subset\Omega$,
and if $\Omega_{j,k}$ is nonempty then its boundary is a real analytic variety of
positive codimension.
Moreover, the integrands vanish identically outside some bounded set.

If $\Omega'$ is a relatively compact measurable subset of $\Omega_{j,k}$,
then 
\[\iint_{\Omega'} (-s_j t-\tilde L_j(v) + \rho_j(\bw) +s_k t + \tilde L_k(v)-\rho_k(\bw)) \,d\bw\]
is an affine function of $t\in\reals$, hence is a smooth function of $t$.
Moreover, according to \eqref{SLip},
the supremum norm of the derivative, with respect to $t$, of
the function defined by this integral is majorized by $C|\Omega'|$.

For each $\delta>0$, there exists a partition of $\Omega$
into sets $\Omega'_{j,k}\subset\Omega_{j,k}$,
$S$, and $T$,
with $\Omega'_{j,k}=\Omega'_{j,k}(\delta)$ a compact subset of $\Omega_{j,k}$,
$|S|\le o_\delta(1)$,
and $\lambda(v,\bw,t)\equiv 0$
whenever $|t-\bart|\le\delta$ and $(v,\bw)\in T$.
Moreover, these can be constructed so that $\Omega'_{j,k}$ increases
and $S,T$ decreases, with respect to set inclusion,
as $\delta$ decreases.
For small $\delta$  and all $t$ satisfying $|t-\bart|\le\delta$,
$K_i(t,0,\cdots,0)$ is a constant multiple of 
$\sum_{j,k} \int_{\Omega'_{j,k}(\delta)} \lambda(v,\bw,t)\,dv\,d\bw$
plus a remainder that is $o_\delta(1)$ in Lipschitz norm.
The principal term is continuously differentiable 
in a small neighborhood of $\bart$, because
$\lambda(v,\bw,t)\equiv b_j(v,\bw,t)-a_k(v,\bw,t)$
for all $(v,\bw)\in \Omega'_{j,k}(\delta)$, for all $t$ sufficiently close to $\bart$.
Because 
\begin{equation}
\Big\|\int_{\Omega'_{j,k}(\delta)} \lambda(v,\bw,t)\,dv\,d\bw
-\int_{\Omega'_{j,k}(\delta')} \lambda(v,\bw,t)\,dv\,d\bw\Big\|_{\text{Lip}}
=O(|\Omega'_{j,k}(\delta)\symdif \Omega'_{j,k}(\delta')|),
\end{equation}
because these domains of integration $\Omega'_{j,k}(\delta)$ are nested with respect to $\delta$,
and because they are all subsets of a fixed bounded set,
these main terms converge in the $C^1$ norm as $\delta\to 0$.
This concludes the proof of Lemma~\ref{lemma:diff2}.
\end{proof}

\section{A series of reductions}

\subsection{A reduction exploiting monotonicity}

\begin{lemma}\label{lemma:flow}
Let $d\in\naturals$. For $j\in J$, let $E_j\subset \reals^d$ be a bounded Lebesgue measurable set. Then, there exist mappings $[0,1]\ni t\mapsto E_j(t)$ of equivalence classes of Lebesgue measurable subsets of $\reals^d$ such that:
\begin{enumerate}
\item $E_j(0)=E_j$ and $E_j(1)=E_j^\star$.
\item $|E_j(t)|=|E_j|$ for all $t\in[0,1]$.
\item $|E_j(s)\Delta E_j(t)|\rightarrow0$ as $s\rightarrow t$.
\item The function $t\mapsto \scriptt_\scriptl(\bE(t))$ is continuous and nondecreasing on $[0,1]$.
\end{enumerate}
\end{lemma}

This is proved in the same way as its analogue for the special case of the Riesz-Sobolev inequality,
using results developed for instance in \cite{BLL},\cite{liebloss}.
See \cite{christRSult}.

\begin{lemma}
In order to prove Theorem~\ref{thm:stability},
it suffices to show that for each $\scriptl,\be$ satisfying
the hypotheses, there exists $\delta_0>0$ such that the conclusion of the theorem holds 
under the additional hypothesis that
\begin{equation}\label{eq:smalldelta}
\dist(\bE,\scripto(\bE^\star))\leq \delta_0\max_j|E_j|.
\end{equation}
\end{lemma}

This is an immediate corollary, as in \cite{christRSult},\cite{christBLL}, 
of the monotonicity and continuity of the flow of the preceding lemma.

\subsection{Reduction to Perturbations Near the Boundary} \label{subsect:dagger}

Let $E_j\subset\reals^d$ $(j\in J)$ be bounded Lebesgue measurable sets with $|E_j|=e_j$ with $\{e_j\}_{j\in J}\subset(0,\infty)^{|J|}$, where $(\scriptl,\be)$ is strictly admissible. Let $\delta=\dist(\bE,\scripto(\bE))$ and choose 
$\psi\in\Sl(d)$ and $v\in\reals^{md}$ to satisfy
\begin{equation} \max_j|\tilde{E_j}\Delta E_j^\star|\leq2\delta, \end{equation}
where $\tilde{E_j}=\psi(E_j)+L_j(v)$. 
We replace $E_j$ with $\tilde{E_j}$, as this doesn't affect the inequality we are trying to prove.

For each $j$, let $B_j=E_j^\star$ and define $f_j$ by
\begin{equation}
\one_{E_j}=\one_{E_j^\star}+f_j=\one_{B_j}+f_j.
\end{equation}

We have the first-order expansion
\begin{equation}
\scriptt_\scriptl(\bE)=\scriptt_\scriptl(\one_{B_j}+f_j:j\in J)=\scriptt_\scriptl(\bE^\star)+\sum_j\langle K_j,f_j\rangle+O(\delta^2),
\end{equation}
where the $K_j$ are defined as in (\ref{eq:Kjdefn}). 
Recalling that $K_j$ is radially symmetric,
we abuse notation by writing
$K_j(r_j)$ as shorthand for $K_j(x)$ where $|x|=r_j$.
Since $\int f_j=0$  and the two functions
$K_j(x)- K_j(r_j)$ and $-f_k$ are both nonnegative on $B_j$ and nonpositive on $\reals^{d}\setminus B_j$,
\begin{equation}
\langle K_j,f_j\rangle =\int (K_j(x)-K_j(r_j))f_j(x)dx
=-\int |K_j(x)-K_j(r_j)|\cdot|f_j(x)|dx.
\end{equation}

Let $\lambda>0$ be a large positive constant which will be chosen independently of $\delta$. We assume $\lambda\delta$ is bounded above by a small positive constant dependent only on $\be$, which is possible 
because we showed in the previous subsection that $\delta$ can be chosen to satisfy
$\delta\leq C\max_j|E_j|$ for some fixed $C<\infty$. By Lemma \ref{lemma:differentiability},
\begin{equation} \label{nonzerogradpayoff}
\langle K_j,f_j\rangle \leq -c\lambda\delta\int_{||x|-r_j|\geq\lambda\delta}|f_j(x)|dx
=-c\lambda\delta|\{x\in E_j\Delta B_j:||x|-r_j|\geq\lambda\delta\}|.
\end{equation}

In this subsection, we will reduce matters to the case in which $E_j\Delta B_j$ is contained entirely in 
$\{x:\big|\,|x|-r_j\,\big|\leq\lambda\delta\}$ for each $j$. 
It was shown in \cite{christRSult} that for each $j\in J$, there exists a set $E_j^\dagger$ with the following properties:
\begin{equation} \label{Edaggerproperties}
\left\{ \begin{gathered}
|E_j^\dagger|=|E_j|
\\
E_j\Delta B_j \text{ is the disjoint union of $E^\dagger_j\Delta B_j$ and $E_j\Delta E_j^\dagger$ }
\\
E_j^\dagger\Delta B_j\subset\{x\in E_j\Delta B_j:\big|\,|x|-r_j\,\big|\leq \lambda\delta\}
\\
|E_j^\dagger\Delta E_j|\leq2|\{x\in E_j\Delta B_j: \big|\,|x|-r_j\,\big|> \lambda\delta\}
\end{gathered} \right. \end{equation}

Write $\bE^\dagger=(E_j^\dagger:j\in J)$.

\begin{lemma}\label{lemma:replacewdagger}
Let $d\geq1$ and $\be\in(0,\infty)^{|J|}$, where $(\scriptl,\be)$ is strictly admissible. There exist $\lambda<\infty$ and $\delta_0,c>0$ with the following property. Let $\bE$ be a $J$-tuple of bounded, Lebesgue measurable subsets of $\reals^d$ satisfying $\max_j|E_j\Delta E_j^\star|=\delta\leq\delta_0$. 
Let $\bE^\dagger$ satisfy \eqref{Edaggerproperties}. Then
\begin{equation}\label{eq:replacewdagger}
\scriptt_\scriptl(\bE)\leq\scriptt_\scriptl(\bE^\dagger)-c\lambda\sum_j|E_j\Delta E_j^\star|\cdot\sum_j|E_j\Delta E_j^\dagger|.
\end{equation}
\end{lemma}

This is a simple consequence of \eqref{nonzerogradpayoff},
as in corresponding lemmas in \cite{christRSult},\cite{christBLL}, 
with no new elements. Therefore the details are omitted.
\qed

Lemma~\ref{lemma:replacewdagger} leads to a reduction to
perturbations near the boundaries of the balls $B_j$, in the strong sense that the
symmetric difference between $E_j$ and $B_j$ is contained in a small neighborhood
of the boundary of $B_j$, with a natural measure of smallness. 
There are two possibilities.
On one hand,
if $\max_j|E_j\Delta E_j^\dagger|\geq\tfrac{1}{10}\max_j|E_j\Delta E_j^\star|$ 
then Lemma \ref{lemma:replacewdagger} yields
\begin{equation}
\scriptt_\scriptl(\bE)\leq\scriptt_\scriptl(\bE^\dagger)-c\Big(\sum_j|E_j\Delta E_j^\star|\Big)^2\leq \scriptt_\scriptl(\bE^\star)-c\dist(\bE,\scripto(\bE^\star))^2
\end{equation}
since $\scriptt_\scriptl(\bE^\dagger)\leq\scriptt_\scriptl(\bE^\star)$
by the Brascamp-Lieb-Luttinger inequality.
In this case, the conclusion of the Theorem has already been reached.

On the other hand, if 
\begin{equation} \label{smallsymmdiff}
\max_j|E_j\Delta E_j^\dagger|<\tfrac{1}{10}\max_j|E_j\Delta E_j^\star|
\end{equation}
then Lemma \ref{lemma:replacewdagger} still yields
\begin{equation} \scriptt_\scriptl(\bE)\leq\scriptt_\scriptl(\bE^\dagger).  \end{equation}
\eqref{smallsymmdiff} implies that
$\dist(\bE^\dagger,\scripto(\bE^\star))\geq \tfrac{1}{2}\dist(\bE,\scripto(\bE^\star))$. 
If we can establish the conclusion for $\bE^\dagger$, that is, if we show that
$\scriptt_\scriptl(\bE^\dagger)\leq\scriptt_\scriptl(\bE^\star)-c\dist(\bEstar,\scripto(\bE^\dagger))^2$,
then these bits of information can  be combined to deduce that 
\[\scriptt_\scriptl(\bE)
\le \scriptt_\scriptl(\bE^\dagger)
\le \scriptt_\scriptl(\bE^\star)-c\dist(\bEstar,\scripto(\bE^\dagger))^2 
\le \scriptt_\scriptl(\bE^\star)-c\dist(\bEstar,\scripto(\bE))^2,\] 
as desired.
Thus matters have been reduced to the situation in which $E_j\Delta B_j$ 
is contained entirely in $\{x:\big|\,|x|-r_j\,\big|\leq\lambda\delta\}$ for each index $j\in J$.

\subsection{Reduction to the Boundar(ies)} \label{subsect:boundaries}

Following \cite{christRSult}, we next reduce matters from an analysis of sets in
$\reals^d$, to an analhysis of functions in $\lt(S^{d-1})$ and to compact selfadjoint linear
operators acting on this Hilbert space.
We drop the superscript $\dagger$, and refer to the set $\bE^\dagger$ introduced
in \S\ref{subsect:dagger} simply as $\bE$.
Thus we assume henceforth that
$\dist(\bE, \scripto(\bE^\star))\leq\delta_0$, 
$\max_j|E_j\Delta E_j^\star|\leq 4\dist(\bE, \scripto(\bE^\star))$, and 
\begin{equation}
E_j\Delta E_j^\star\subset\{x:\big|\,|x|-r_j\,\big|\leq\lambda\delta\}
\end{equation}
for a certain large constant $\lambda<\infty$ that depends only on $d,\scriptl,\be$.

Continue to write $B_j=E_j^\star$ and $\one_{E_j}=\one_{B_j}+f_j$, but now refine this representation by writing $f_j=f_j^+-f_j^-$, where $f_j^+=\one_{E_j\setminus B_j}$ and $f_j^-=\one_{B_j\setminus E_j}$. Let $(r,\theta)$ be polar coordinates on $\reals^d$ and define $F_j^\pm\in L^2(S^{d-1})$ by
\begin{equation}\label{define:bigFpm}
F_j^\pm(\theta)=\int_{\reals^+}f_j^\pm(t\theta)t^{d-1}dt.
\end{equation}
Further define 
\begin{equation} \label{defn:bigF}
F_j=F_j^+-F_j^-. 
\end{equation}
Under the hypothesis that $E_j\Delta B_j\subset\{x:||x|-r_j|\leq\lambda\delta\}$,
\begin{equation}\label{eq:reducetoL2norms}
|E_j\Delta B_j|^2\asymp||F_j^+||_{L^2}^2+||F_j^-||_{L^2}^2,
\end{equation}
where $u\asymp v$ means $u\leq Cv$ and $v\leq Cu$ for a constant $C$ depending only on $d$ and $\be$. (We have this dependence in the above since $\lambda$ depends only on $d$ and $\be$.)

Let $\sigma$ denote the rotation-invariant surface measure on $S^{d-1}$, 
normalized so that Lebesgue measure in $\reals^d$ is equal to $r^{d-1}drd\sigma(\theta)$. 
For each $i\neq j\in J$, define quadratic forms $Q_{i,j}$ on $L^2(S^{d-1},\sigma)$ by
\begin{equation}\label{define:Qij}
Q_{i,j}(F,G)=\iint_{S^{d-1}\times S^{d-1}}F(x)G(y)M_{i,j}(r_ix, r_jy)d\sigma(x)d\sigma(y),
\end{equation}
where $M_{i,j}$ are as in \eqref{eq:Lij}.
Formally, $Q_{i,j}(F,G) = \langle T_{i,j}F,G\rangle$
where the inner product is that of $L^2(S^{d-1},\sigma)$
and $T_{i,j}$ is the integral operator associated to the kernel $M_{i,j}(r_i x,r_j y)$.
Since $M_{i,j}$ is bounded and Borel measurable, 
each $T_{i,j}$ is a well-defined bounded linear operator on $\lt(S^{d-1},\sigma)$.
Moreover, $T_{i,j}$ is compact, since the restriction of $M_{i,j}$
to $S^{d-1}\times S^{d-1}$ is bounded and Borel measurable, hence belongs
to $L^2(\sigma\times\sigma)$.

\begin{lemma}\label{lemma:Qijsymmetric}
For each $i\ne j$, $Q_{i,j}$ is a symmetric quadratic form on $L^2(S^{d-1},\sigma)$.
\end{lemma}

That is, $Q_{i,j}(F,G) = Q_{i,j}(G,F)$ for arbitrary real-valued $F,G\in \lt(S^{d-1},\sigma)$.

\begin{proof}
It suffices to show that $M_{i,j}(r_ix, r_jy)=M_{i,j}(r_iy, r_jx)$ whenever $|x|=|y|=1$. 
We claim that $M_{i,j}(Ru,Rv)=M_{i,j}(u,v)$ for any $R\in O(d)$ 
and any $u,v\in \reals^d$.
Then given $x,y\in S^{d-1}$, choose a reflection $R$ satisfying $R(x)=y$ and $R(y)=x$ 
and invoke the claim with $u=r_i y$ and $v = r_j x$ to conclude that
\[M_{i,j}(r_ix, r_jy)=M_{i,j}(R(r_i y),R(r_j x))=M_{i,j}(r_iy, r_jx) .\]

For arbitrary $f_i,f_j$,
\begin{equation}
\iint M_{i,j}(Ru,Rv)f_i(u)f_j(v)\,du\,dv
=\iint M_{i,j}(u,v)f_i(R^{-1}u)f_j(R^{-1}v)\,du\,dv.
\end{equation}
This is equal to $\scriptt_\scriptl(\bg)$,
where $\bg=(g_n:n\in J)$ is defined by $g_i(u)=f_i(R^{-1}u), g_j(v)=f_j(R^{-1}v)$, and $g_k=\one_{B_k}$ for 
every $k\in J\setminus\{i,j\}$.
By the $O(d)$-invariance of $B_k$ and $\scriptt_\scriptl$,
\begin{equation}
\scriptt_\scriptl(\bg)=\scriptt_\scriptl({\bf h})=\iint M_{i,j}(u,v)f_i(u)f_j(v)\,du\,dv,
\end{equation}
where ${\bf h}=(h_n:n\in J)$ is defined by $h_i(u)=f_i(u)$, $h_j(v)=f_j(v)$, 
and $h_k=\one_{B_k}$ for every $k\in J\setminus \{i,j\}$.
\end{proof}

Write $\bF=(F_j: j\in J)$.
Define 
\begin{equation}\label{define:Q}
Q(\bF)=\sum_{i\neq j\in J}Q_{i,j}(F_i,F_j)
\end{equation}
and
\begin{equation} \label{gammadefn}
\gamma_j = |\nabla K_j(x)| \ \text{ for $|x|=r_j$,} 
\end{equation}
recalling that 
$K_j$ is radially symmetric and that we
have shown that its gradient exists on this sphere,
under the hypothesis that $\be$ is strictly admissible together with the standing assumption that $d\ge 2$.

The goal of this subsection is to establish the following second order expansion.

\begin{proposition}\label{prop:exponbdry}
Under the hypotheses introduced at the beginning of \S\ref{subsect:boundaries},
\begin{equation}\label{eq:exponbdry}
\scriptt_\scriptl(\bE)\leq\scriptt_\scriptl(\bE^\star)-\tfrac{1}{2}\sum_{j\in J}\gamma_jr_j^{1-d}(||F_j^+||_{L^2}^2||F_j^-||_{L^2}^2)+Q(\bF)+o(\delta^2).
\end{equation}
\end{proposition}

To prove Proposition \ref{prop:exponbdry}, we 
substitute $\one_{E_j} = \one_{B_j}+f_j$ and expand using the multilinearity of $\scriptt_\scriptl$.
The zeroth order term is equal to $\scriptt_\scriptl(\bE^\star)$. The following three lemmas will address 
first, second, and higher order terms, in that order.

\begin{lemma}\label{lemma:1storder}
For each $j\in J$,
\begin{equation}\label{eq:1storder}
\langle K_j,f_j\rangle\leq -\tfrac{1}{2}\gamma_jr_j^{1-d}(||F_j^+||_{L^2}^2 + ||F_j^-||_{L^2}^2)+o(\delta^2).
\end{equation}
\end{lemma}

The proof is essentially identical  to the proofs of the corresponding
lemmas in \cite{christRSult},\cite{christBLL}, so the details are omitted.
Whereas a remainder term $O(\delta^3)$ was obtained in those sources,
here only a weaker bound $o(\delta^2)$ results,
because here the derivative of $K_j$ is merely known to be continuous,
while there it was Lipschitz.
\qed

\begin{lemma}\label{lemma:2ndorder}
Let $i\ne j\in J$.
Let $F_i,F_j\in\lt(S^{d-1})$ be the functions associated to $f_i,f_j$, respectively, by \eqref{defn:bigF}.
Then
\begin{equation}\label{eq:2ndorder}
\iint M_{i,j}f_if_j = Q_{i,j}(F_i,F_j)+O(\delta^{3}).
\end{equation}
\end{lemma}

\begin{proof}
This follows from Lemma~\ref{lemma:Holder2},
by the same reasoning as shown for $m=2$ in the corresponding lemma in \cite{christBLL}.
\end{proof}

\begin{lemma}\label{lemma:3rdorder}
Let $\bg=(g_n:n\in J)$ be a $J$-tuple of functions  such that for each $j\in J$, 
either $g_j=f_j$ or $g_j=\one_{B_j}$. Suppose that $g_j=f_j$ for at least three distinct indices $j$. Then,
\begin{equation}\label{eq:3rdorder}
\Phi(\bg)=O(\delta^3).
\end{equation}
\end{lemma}

\begin{proof}
Let $i,j,k$ be three distinct indices 
such that $g_i=f_i, g_j=f_j, g_k=f_k$. 
Consider first the case in which $\{L_i,L_j,L_k\}$ is linearly independent. 
Necessarily, then, $m\ge 3$.
After a change of variables, we may write
\begin{equation}
\Phi(\bg)=\iiint_{(\reals^d)^3}f_i(x)f_j(y)f_k(z)F(x,y,z)\,dx\,dy\,dz,
\end{equation}
where $F(x,y,z)$ has finite $L^\infty$ norm majorized by a function of  $\scriptl$, $\be$, 
$d$, and $\{i,j,k\}$; $F(x,y,z)$ is obtained by integrating $\prod_{n\in J\setminus\{i,j,k\}} L_n^d(\bx)$
over a suitable translate of $(\reals^d)^{m-3}$. 
Majorizing $F$ by its $L^\infty$ norm, the integral becomes a simple product and it follows that
\begin{equation}
\Phi(\bg)\leq||f_i||_{L^1}||f_j||_{L^1}||f_k||_{L^1}||F||_{L^\infty}=O(\delta^3).
\end{equation}

Suppose instead that $\{L_i,L_j,L_k\}$ is linearly dependent.
For any two distinct indices $n,n'$, $L_n^1,L_{n'}^1$ 
are linearly independent by the nondegeneracy hypothesis. From the structural hypothesis
it follows that $L_n^d,L_{n'}^d$ are likewise linearly independent.
Thus $\{L_i,L_j,L_k\}$ spans a two-dimensional space. 
After a change of variables and after pulling out the $L^\infty$ norm
of a factor analogous to the function $F$ in the preceding paragraph,
\begin{equation}\label{eq:useRS}
|\Phi(\bg)| \le C \iint_{(\reals^d)^2}|f_i(x)|\,|f_j(y)|\,|f_k(ax+by)|\,dx\,dy ,
\end{equation}
where $a,b\neq 0$, and  $C$
is majorized by a finite quantity depending only on $\scriptl$, $\be$, $d$, and $\{i,j,k\}$. 

It is no longer true that this integral is majorized  by a constant multiple of the product of the $L^1(\reals^d)$ 
norms of the three factors of the integrand for arbitrary functions $f_n$. Nonetheless,
the integral is $O(\delta^3)$.
Indeed, each factor $f_n$ satisfies $\norm{f_n}_{L^\infty} \le 1$,
and each is supported on an annular subset of $\reals^d$ centered at $0$, 
having inner and outer radii $r_n\pm O(\delta)$.
Introduce polar coordinates $x=(\rho_i,\theta_i)$ and $y=(\rho_j,\theta_j)$.
Fix any $\rho_i,\rho_j$ with distance $C\delta$ of $r_i,r_j$, respectively.
By Lemma~\ref{lemma:smallmeasure},
the set of ordered pairs $(\theta_i,\theta_j)$ for which $\big|\,|ax+by|-r_k\,\big|\le C\delta$
has $\sigma\times\sigma$ measure $O(\delta)$.
\end{proof}

\section{Spectral problem and balancing}

Write $\bF=(F_j: j\in J)$ and $\pi_n(\bF) = (\pi_n(F_j): j\in J)$.
We say that $\bF\in L^2$ if each component $F_j$ belongs to $L^2(S^{d-1},\sigma)$.

In light of Proposition \ref{prop:exponbdry}, we consider the optimal constant $A$ in the inequality
\begin{equation}\label{eq:introA}
Q(\bF)\leq A\sum_{j\in J}\gamma_j r_j^{1-d}\norm{F_j}_{L^2}^2
\ \ \text{for every } \bF\in L^2(S^{d-1})
\text{ satisfying $\pi_0(\bF)=0$.}
\end{equation}
If it were known that $A<\frac{1}{2}$, then combining \eqref{eq:reducetoL2norms} with the inequality
\begin{equation}
||F_j||_{L^2}^2=||F_j^+||_{L^2}^2+||F_j^-||_{L^2}^2-2\langle F_j^+,F_j^-\rangle\leq||F_j^+||_{L^2}^2+||F_j^-||_{L^2}^2
\end{equation}
and Proposition~\ref{prop:exponbdry} would yield the conclusion of Theorem~\ref{thm:stability}.
However, 
this optimal constant cannot be strictly less than $\tfrac12$.
Indeed, this strict inequality together with the machinery developed above would imply that
all maximizers are tuples of balls centered at the origin, contradicting the
existence of a large family of symmetries of the inequality.
The assumption that $\max_j|E_j\Delta E_j^\star|$ is comparable to $\dist(\bE,\scripto(\bE^\star))$
must be exploited in order to compensate for these symmetries.

We recast \eqref{eq:introA}
in terms of spherical harmonics. For $\nu\geq0$, let $\scripth_\nu \subset L^2(S^{d-1})$ 
denote the subspace of spherical harmonics of degree $\nu$. 
$L^2(S^{d-1})$ is naturally identified with $\oplus_{\nu=0}^\infty\scripth_\nu$. 
Denote by $\pi_\nu$ the orthogonal projection from $L^2(S^{d-1})$ onto $\scripth_\nu$.

$Q$ commutes with rotations, a consequence of the symmetry hypothesis \eqref{association}.
By Lemma \ref{lemma:Qijsymmetric}, 
for each $i\neq j\in J$, $Q_{i,j}$ takes the form
$Q_{i,j}(F,G)\equiv\langle T_{i,j}(F),G\rangle$, 
for a certain compact, selfadjoint linear operator $T_{i,j}$ on $L^2(S^{d-1})$.
Moreover, $T_{i,j}$ maps $\scripth_\nu$ to itself for every $\nu\in\naturals$, 
and $T_{i,j}$ agrees with a scalar multiple $\lambda(\nu,\be)$ of the identity operator on $\scripth_\nu$. 
Therefore
\begin{equation} Q(\bF)=\sum_{\nu=1}^\infty Q(\pi_\nu(\bF)).  \end{equation}
The summation begins at $\nu=1$ since $\int f_j=0$ implies $\int_{S^{d-1}}F_j\,d\sigma=0$, 
whence $\pi_0(\bF)=0$. 

The next lemma is an immediate consequence of the compactness and selfadjointness
of the operators $T_{i,j}$. 

\begin{lemma}\label{lemma:Tkcompactness}
Let $(\scriptl,\be)$ be strictly admissible.
For each $d\geq2$, there exists a sequence $\Lambda_\nu=\Lambda_\nu(\scriptl,\be)$ 
satisfying $\lim_{\nu\rightarrow\infty}\Lambda_\nu=0$ such that for each $\nu\geq1$ 
and all $\bG\in \scripth_\nu^J$,
\begin{equation}
|Q(\bG)|\leq\Lambda_\nu\sum_j||G_j||_{L^2(S^{d-1})}^2.
\end{equation}
\end{lemma}

As in \cite{christRSult}, we are not able to compute the eigenvalues of 
the operators $T_{i,j}$, which depend on the parameters $\scriptl,\br,d,\nu$. 
They can be computed for the Riesz-Sobolev inequality
for $d=2$. That computation reveals that they lack useful monotonicity properties with
respect to the degree $\nu$.
Therefore we are led to develop an indirect analysis, following \cite{christRSult}, 
that works more holistically
with $Q(\bF)-\tfrac12 \sum_j \gamma_j r_j^{1-d} \norm{F_j}_{\lt}^2$.

As in related analyses, it will be essential to eliminate certain spherical harmonic components.

\begin{lemma}\label{lemma:balancing}
Let $d,m\geq2$. Let $\scriptl$ be nondegenerate, and let $(\scriptl,\be)$ be strictly admissible.
Let $\bE$ be as above, with $\delta=\dist(\bE,\scripto(\bE^\star))$. 
Let $J'\subset J$ have cardinality $m$,
and suppose that $\{L_j: j\in J'\}$ is linearly independent.
Let $n\in J'$.
There exist $\bv\in(\reals^d)^m$ and an invertible, measure preserving linear transformation 
$\psi$ of $\reals^d$ such that
\newline
(i) $|\bv|=O(\delta)$ and $||\psi-I||=O(\delta)$.
\newline
(ii) The functions $\tilde{F_j}$ associated to the sets $\tilde{E_j}=\psi(E_j)+L_j(\bv)$ satisfy
\begin{equation}
\left\{\begin{aligned}
\pi_1(\tilde{F_j})& =0,\  \text{for all $j\in J'$}
\\
\pi_2(\tilde{F_n})& =0.
\end{aligned}\right.
\end{equation}
\end{lemma}

The term ``associated'' means that
$\tilde F_j$ is constructed from $\tilde E_j$ in the same way that $F_j$ was constructed from $E_j$:
$\tilde F_j(\theta) = \int_0^\infty (\one_{\tilde E_j}-\one_{B_j})(t\theta)\,t^{d-1}\,dt$.
The norm $\norm{\psi-I}$ is defined by choosing any fixed norm on the vector space
of all $d\times d$ real matrices, expressing the elements $\psi,I$
of the general linear group as such matrices, and taking the norm of the
difference of the two matrices.

\begin{proof}[Proof of Lemma~\ref{lemma:balancing}]
It is shown in \cite{christRSult} that there exist
$\psi\in\Sl(d)$ and $w\in\reals^d$
such that 
$\norm{\psi-I}+ |w| \le C_\lambda |E_n\symdif B_n|$,
and $\tilde E_n = \psi(E_n)+w$ satisfies $\pi_\nu(\tilde E_n)=0$
for $\nu=1,2$
and $\tilde E_n\symdif B_n$ 
is contained in a $C_\lambda |E_n\symdif B_n|$--neighborhood of the boundary of $B_n$.

Likewise, it is proved in \cite{christRSult} that for each $j\in J$
there exists $v_j\in\reals^d$ such that $|v_j|\le C_\lambda |E_j\symdif B_j|$,
$\tilde E_j = E_j+v_j$
is contained in a $C_\lambda |E_j\symdif B_j|$--neighborhood of the boundary of $B_j$,
and $\pi_1(\tilde E_j)=0$.

Define $\bv\in (\reals^d)^m$ to be the unique vector satisfying $L_j(\bv)=v_j$
for all $j\in J'\setminus\{n\}$, and $L_n(\bv)=w$.
Then $\psi,\bv$ have all of the desired properties.
\end{proof}

Fix $J'\subset J$  satisfying the hypothesis of Lemma~\ref{lemma:balancing}, and fix $n\in J'$.
A $J$-tuple $\bG=(G_j: j\in J)$ of
spherical harmonics of some common degree $\nu$
is said to be {\em balanced} if $\nu\geq3$, or if $\nu=2$ and $G_n=0$,
or if $\nu=1$ and $G_j=0$ for every $j\in J'$.

By the reductions made thus far, Theorem~\ref{thm:stability} will follow from the next result,
whose proof occupies the remainder of the paper.

\begin{proposition}\label{prop:end}
Let $d\geq2$ and let $(\scriptl, \be)$ satisfy the hypotheses of Theorem \ref{thm:stability}. 
There exists $A<\frac{1}{2}$ such that for every balanced $J$-tuple 
$\bG$ of spherical harmonics, 
\begin{equation}
Q(\bG)\leq A\sum_j\gamma_jr_j^{1-d}||G_j||_{L^2}^2.
\end{equation}
\end{proposition}

\section{Algebraic preliminaries}

For each $d\in\naturals$ define
\begin{equation} \Lambda_{d} = \{\by\in (\reals^{d})^J: \text{ there exists $\bx\in \reals^{md}$
satisfying $L_j^d(\bx)=y_j$  for every } j\in J\}.
\end{equation}
$\Lambda_d$ is a linear subspace of $(\reals^d)^J$ of dimension $md$.

\begin{lemma}\label{lemma:affine}
Let $k\ge 1$.
Let $(\varphi_j: j\in J)$ be a tuple of measurable functions
$\varphi_j: \reals^{k}\to\reals$ such that
for almost every $\by\in\Lambda_{k}$,
\begin{equation} (\varphi_j(y_j): j\in J)\in \Lambda_1. \end{equation}
Then each $\varphi_j$ agrees almost everywhere on $\reals^k$
with an affine function.
\end{lemma} 

\begin{proof} 
Choose a subset $J'\subset J$ of cardinality $m$ such that
$\{L_j^1: j\in J'\}$ is linearly independent. 
Then  $L_j^k: (\reals^k)^m\to\reals^k$, and
$\{L_j^k: j\in J'\}$ is likewise linearly independent.

Identify $J'$ with $\{1,2,\dots,m\}$,
and identify $J\setminus J'$ with $\{m+1,m+2,\dots,|J|\}$.
Let $h:\reals^m\to\reals^m$ be an invertible linear transformation
satisfying $L_j^1\circ h(\bx)=x_j$ for each $j\in J'$.
Change variables by $\bx\mapsto h(\bx)$, so that henceforth,
$L_j^1(\bx)=x_j$ for each $j\in J'$.
Then
$L_j^k(\bx)=x_j$ for each $j\in J'$, for every $\bx\in (\reals^k)^m$.

There exist linear mappings $\psi_j^1: \reals^m\to\reals^1$ such that
$\Lambda_1$ is identified, in these new variables, with
$\{\bx\in\reals^{J}: x_j=\psi_j^1(x_1,\dots,x_m)
\text{ for all }j\in J'\}$. 
Likewise, $\Lambda_k$ is identified with
$\{\bx\in(\reals^k)^{J}: x_j=\psi_j^k(x_1,\dots,x_m)
\text{ for all }j\in J'\}$. 
It is more convenient
for our purpose to employ an alternative description of $\Lambda_1$:
for certain coefficients $a_{i,j}\in\reals$, 
$\Lambda_1$ is equal to the set of all $\bx\in\reals^J$ that satisfy $x_i = \sum_{j=1}^m a_{i,j}x_j$
for all $i>m$. 
In these terms, the hypothesis of the lemma  
is that for almost every $\by\in(\reals^{k})^m$, for each $i>m$,
\begin{equation}\label{eq:funceq}
\varphi_i(\psi_i(y_1,\dots,y_m)) = \sum_{j=1}^m a_{i,j}\varphi_j(y_j) \ \text{ for every }  i>m.
\end{equation}

Consider any particular $i>m$.
By hypothesis, the mapping $\reals^m\owns (y_1,\dots,y_m)\mapsto \psi_i(y_1,\dots,y_m)\in\reals^1$
is not a scalar multiple of any of the $m$ mappings $(y_1,\dots,y_m)\mapsto y_j$.
Fix any $i>m$.
A well-known result states that under this condition, 
for any solution $(\varphi_1,\dots,\varphi_m,\varphi_i)$ of \eqref{eq:funceq},
each of the functions $\varphi_1,\dots,\varphi_m,\varphi_i\circ\psi_i$
must agree almost everywhere on $\reals^k$ with an affine function.
From the surjectivity and linearity of $\psi_i$, it follows immediately that
$\varphi_i$ itself must likewise be affine.
\end{proof}

Let $\bG=(G_j: j\in J)$ be a $J$-tuple of spherical harmonics of some
common degree $\nu$.
To each $G_j$ associate a polynomial $G_{j,o}$, defined as follows:
Express $G_j$, regarded as a homogeneous polynomial defined on $\reals^d$, 
as a linear combination of monomials in $(x_1,\dots,x_d)$.
Such an expansion is unique.
Let $G_{j,o}$ be the sum of all terms that have odd degrees with respect to $x_d$.
For $x'\in\reals^{d-1}$ define 
\begin{equation}\label{eq:assocP}
P_j(x')=r_j^{2-d-n}x_d^{-1}G_{j,o}(x',x_d)\text{ with }x_d=(r_j^2-|x'|^2)^{1/2}
\end{equation}
More precisely,
$x_d^{-1}G_{j,o}(x',x_d)$ is a linear combination of monomials in $(x',x_d^2)$.
Substituting $r_j^2-|x'|^2$ for $x_d^2$ defines $P_j(x')$.
Thus $P_j$ is a real-valued polynomial, with domain $\reals^{d-1}$.

Define $P_\sharp(\bG): \Lambda_{d-1}\to [0,\infty)$ by 
\begin{equation} \label{eq:Psharpdefn} 
P_\sharp(\bG)(\theta)^2 =\dist^2((P_j(\theta))_{j\in J},\Lambda_1).\end{equation}
This is the distance from the vector $(P_j(\theta): j\in J)$ in $\reals^J$
to a certain linear subspace of $\reals^J$, so it is the restriction to $S^{d-1}$
of a polynomial of degree $2\nu$.

\begin{lemma}\label{lemma:PsharpOG}
Let $d\geq2$. For any nonzero balanced 
$\bG\in \scripth_\nu^J$,
there exists $A\in O(d)$ such that $P_\sharp(A(\bG))\neq0$.
\end{lemma}

\begin{proof}[Proof of Lemma \ref{lemma:PsharpOG}]
It is shown in \cite{christRSult} that for
any spherical harmonic $G$ of any degree $\ge 3$, 
there exists $A\in O(d)$ such that the
polynomial $P$ associated to $G\circ A$ by \eqref{eq:assocP} is not affine.

If $P_\sharp(\bG)\equiv 0$ then each according to Lemma~\ref{lemma:affine},
$P_j$ is affine for every $j\in J$. Thus if $P_\sharp(A(\bG))
\equiv 0$ for every $A\in O(d)$
then the polynomial $P_j$ associated to $G_j\circ A$ must vanish identically for every $j\in J$
and every $A\in O(d)$. Thus $\bG\equiv 0$.

Consider the case in which $\nu=2$.
The condition that $\bG$ is balanced becomes $G_m\equiv 0$.
Therefore the condition that $P_\sharp(\bG)\equiv 0$,
according to \eqref{eq:funceq}, means that for each $i>m$,
$P_i\circ \psi_i(y_1,\dots,y_m))$
can be expressed as $\sum_{j=1}^{m-1} c_{i,j} P_j(y_j)$,
where $\psi_i$ is the linear mapping introduced in the discussion around \eqref{eq:funceq}.
Because the invertible linear transformation $\psi_i(y_1,\dots,y_m)$ is not independent of $y_m$,
this functional relation implies that $P_i$, as well as $P_j$ for all $y\le m-1$, are constant.
A lemma in
\cite{christRSult} states that 
if $G$ is a spherical harmonic of degree $\nu=2$
for which the polynomial $P$ associated to $G\circ A$ by \eqref{eq:assocP}
is constant for every $A\in O(d)$, 
then $G\equiv 0$.

For $\nu=1$,
balancing means that $G_j\equiv 0$ for all $j\le m$, and
\eqref{eq:funceq} becomes
$P_i\circ\psi_i\equiv 0$ for all $i>m$. 
Thus if $P_\sharp(A(\bG))\equiv 0$ for every $A\in O(d)$
then for every $A\in O(d)$ and every $i\in J$,
the polynomial  $P_i$ associated to $G_j\circ A$
vanishes identically.
Therefore for any $A\in O(d)$, $G_{i,o}\equiv 0$.
Since $G_i$ has degree $1$, this means that $G_i\circ A$ is independent of $x_d$.
This can hold for every $A\in O(d)$ only if $G_i\equiv 0$.
\end{proof}

\section{Final Steps}
Proposition~\ref{prop:end} states that 
\begin{equation} \label{eq:conclusionrestated}
\text{There exists $A<\tfrac12$ such that }
Q(\bG)\leq A\sum_j\gamma_jr_j^{1-d}||G_j||_{L^2}^2
\end{equation}
for every balanced $\bG\in\scripth_\nu^J$, for every $\nu\ge 1$.
We are now in a position to complete its proof.
It suffices to prove that for each $\nu\ge 1$ there exists $A<\tfrac12$,
possibly dependent on $\nu$, for which this holds;
Lemma~\ref{lemma:Tkcompactness} then yields uniformity in $\nu$.
Moreover, since the space of spherical harmonics of degree $\nu$ has finite dimension,
a simple compactness argument shows that
it suffices to prove this with $A$ depending on $\bG$.

Let $\nu \ge 1$, and consider any such $\bG$.
There exists $A\in O(d)$ satisfying $P_\sharp(A(\bG))\ne 0$.
The conclusion of Proposition~\ref{prop:end} is invariant under the diagonal action $\bG\mapsto A(\bG)$
of $O(d)$.  According to Lemma~\ref{lemma:PsharpOG},
there exists $A\in O(d)$ suxch that $P_\sharp(\bG)\ne 0$.
Therefore it suffices to prove \eqref{eq:conclusionrestated} under the additional
hypothesis that $P_\sharp(\bG)$ does not vanish identically.

For each $j\in J$, define $\phi_j:S^{d-1}\times(-\frac{1}{2},\frac{1}{2})\rightarrow\reals^+$ 
by the equations
\begin{equation}
\int_{r_j}^{r_j+\phi_j(\theta,s)}t^{d-1}dt=sG_j(\theta).
\end{equation}
Then $\phi_j(\theta,s)$ has the same sign as $sG_j(\theta)$.
As noted in \cite{christRSult}, 
\begin{equation} r_j^{d-1}\phi_j(\theta,s)=sG_j(\theta)+O(s^2).  \end{equation}

For $s\in(-\frac{1}{2},\frac{1}{2})$, define sets $E_j(s)\subset \reals^d$ by
\begin{equation}
E_j(s)=\{t\theta:t\leq r_j+\phi_j(\theta,s)\}.
\end{equation}
Define $\bE(s)=\bE_\bG(s)=(E_j(s):j\in J)$. Since $\int_{S^{d-1}}G_jd\sigma=0$, $|E_j(s)|=|E_j|$ for all $s$. The function $F_{j,s}$ associated to $E_j(s)$ depends smoothly on $(\theta,s)$ and satisfies $F_{j,s}=sG_j+O(s^2)$.

We will exploit a variant of Proposition~\ref{prop:exponbdry}.
\begin{lemma}\label{lemma:likeexponbdry}
Let $d\geq2$, $\nu\in \naturals$, and $(\scriptl,\be)$ satisfy the hypotheses of Theorem \ref{thm:stability}. 
There exist $c,\eta>0$ such that uniformly for all $J$-tuples $\bG$ 
of spherical harmonics of degree $\nu$ satisfying $||\bG||=1$, 
\begin{equation}\label{eq:likeexponbdry}
\scriptt(\bE(s))=\scriptt(\bE^\star)-\tfrac{1}{2}s^2\sum_{j\in J}\gamma_kr_k^{1-d}||G_j||_{L^2}^2+s^2Q(\bG)+o(s^2)
\end{equation}
whenever $|s|\leq\eta$.
\end{lemma}

Lemma~\ref{lemma:likeexponbdry} and Proposition~\ref{prop:exponbdry} are closely related,
but differ in two essential respects. 
The lemma is only concerned with tuples $\bE(s)$ of a special type, but it yields
a stronger conclusion, with equality up to the indicated remainder rather than with a one-sided inequality.
We will analyze its left-hand side by other means, will use this to gain control of its right-hand
side, and then will use that information to prove \eqref{eq:conclusionrestated}.

The proof of Lemma~\ref{lemma:likeexponbdry} is essentially identical to the proof
of the corresponding lemma in \cite{christRSult}, so the details are omitted. \qed

The sets $E_j(s)$ enjoy three properties that are not shared by general
sets $E_j$. Firstly, for all $s$ sufficiently close to $0$,
the functions $F_j^\pm(s)$ are continuous and $F_j^+(s),F_j^-(s)$ are nonzero on disjoint sets.
Therefore $\norm{F_j(s)}_{\lt}^2 = \norm{F_j^+(s)}_{\lt}^2 + \norm{F_j^-(s)}_{\lt}^2$.
Secondly, for each $\theta\in S^{d-1}$, 
$\{t\in\reals^+:t\theta\in E_j\setminus B_j\}$ is an interval whose left endpoint equals $r_j$, 
and $\{t\in\reals^+:t\theta\in B_j\setminus E_j\}$ is an interval whose right endpoint equals $r_j$. 
This follows from the smoothness of spherical harmonics and the fact that 
any $C^N$ norm of a spherical harmonic is majorized by a constant multiple of its $\lt$ norm,
with the constant depending only on degree $\nu$ and dimension $d$.
Thirdly,
\begin{equation}
\langle K_j,f_j(s)\rangle = -\tfrac12 \gamma_j r_j^{1-d} \norm{F_j(s)}_{\lt}^2 + o(\delta^2).
\end{equation}
This is proved just as in \cite{christRSult}.

\begin{lemma}\label{lemma:Es}
Let $d\ge 2$.
Suppose that $(\scriptl,\be)$ satisfies the hypotheses of Theorem~\ref{thm:stability}. 
For each $\nu\in\naturals$ there exist $c,\eta>0$, depending on $\nu, d, \be$, 
such that for every balanced 
$\bG\in\scripth_\nu^J$
satisfying $||\bG||=1$,
\begin{equation}\label{eq:Es}
\scriptt(\bE(s))\leq\scriptt(\bE^\star)-cs^2
\ \text{ whenever $|s|\le\eta$.}
\end{equation}
\end{lemma}

Proposition~\ref{prop:end} follows from 
Lemmas~\ref{lemma:Es} and \ref{lemma:likeexponbdry}.
Indeed,
subtracting \eqref{eq:Es} from \eqref{eq:likeexponbdry}, dividing by $s^2$, and rearranging terms, gives
\begin{equation}
-\tfrac{1}{2}\sum_j\gamma_jr_j^{1-d}||G_j||_{L^2}^2+Q(\bG)\leq -c||\bG||^2+O(s)\leq -c'||\bG||^2
\end{equation}
for sufficiently small $s$, for some $c,c'>0$. 
This is equivalent to \eqref{eq:conclusionrestated}.
\qed

Thus it remains to prove Lemma~\ref{lemma:Es}.
For each $j\in J$ and $x'\in\reals^d$, define 
\begin{equation} I_j(x',s)=\{t\in\reals:(x',t)\in E_j(s)\}. \end{equation}
Write $\bI(x',s) = (I_j(x',s): j\in J)$.
Denote by $I_j^\star(y',s)\subset\reals$ the closed interval centered at $0$,
whose length is equal to the length of $I_j(y',s)$, and $\bI^\star(\by',s)
= (I_j^\star(y'_j,s): j\in J)$.
The Steiner symmetrization $\bI(s)^\dagger$ is
the set of all $(\by',\bt)$ such that $\bt\in \bI(\by',s)^\star$. 

Rewrite $\scriptt_\scriptl(\bE(s))$ as
\begin{align*}  
\scriptt_\scriptl(\bE(s)) 
& =\int_{\Lambda_d}\prod_j\one_{E_j(s)}(y_j) \,d\lambda_d(\by)
\\&
=\int_{\Lambda_{d-1}}\scriptt_{\scriptl^1}(\bI(\by',s)) \,d\lambda_{d-1}(\by')
\\&
= \scriptt_\scriptl(\bE(s)^\dagger) - 
\int_{\Lambda_{d-1}}
\Big[ \scriptt_{\scriptl^1}(\bI(\by',s)^\star) - \scriptt_{\scriptl^1}(\bI(\by',s)) \Big]
\,d\lambda_{d-1}(\by')
\\&
\le  \scriptt_\scriptl(\bEstar)
- \int_{\Lambda_{d-1}}
\Big[ \scriptt_{\scriptl^1}(\bI(\by',s)^\star) - \scriptt_{\scriptl^1}(\bI(\by',s)) \Big]
\,d\lambda_{d-1}(\by')
\end{align*}
Here $\by'\in(\reals^{k-1})^J$, and $\lambda_k$ is an appropriately normalized Lebesgue measure on 
the linear subspace $\Lambda_k$. 
In the final line, we have invoked the Brascamp-Lieb-Luttinger inequality
to conclude that $\scriptt_\scriptl(\bE(s)^\dagger) \le \scriptt_\scriptl(\bE^\star)$.
That same inequality guarantees that the integrand in the final line is nonnegative.

In order to conclude that
$\scriptt_\scriptl(\bE(s)) \le  \scriptt_\scriptl(\bEstar) -cs^2$,
then,
it suffices to show that this integrand is $\ge c's^2$
on some subset of $\Lambda_{d-1}$ whose $\lambda_{d-1}$ measure
is bounded below by a positive constant independent of $s$.
We will do this by invoking the $d=1$--dimensional case of Theorem~\ref{thm:stability},
which is the main result of \cite{christBLL}.\footnote{We invoke the one-dimensional
case of Theorem~\ref{thm:stability} only for intervals, rather than for general measurable
sets.} 
Thus we next verify that hypotheses of that result are satisfied.
These hypotheses are nondegeneracy of $\scriptl^1$, strict admissibility of $(\scriptl^1,\be(\by',s))$,
where $e_j(\by',s) = |I_j(\by',s)|$,
and genericity of $(\scriptl^1,\be(\by',s))$.  
Nondegeneracy of $\scriptl^1$ is one of the hypotheses of Theorem~\ref{thm:stability},
so no discussion of it is needed.
Strict admissibility of $\be(\by')$ is discussed in the next paragraph.
The genericity hypothesis has not yet been encountered in the present paper, and is discussed below. 

$(\scriptl^1,\be(\by',s))$
is strictly admissible relative to $\scriptl^1$
for $(\by',s)=(\bzero,0)$. Indeed,
$e_j(\bzero,0) = \omega_d^{-1/d}|E_j|^{1/d}=\omega_d^{-1/d}e_j^{1/d}$. 
Strict admissibility of $\omega_d^{-1/d}\be^{1/d}$ is equivalent
to strict admissibility of $(\scriptl^1,\be^{1/d})$, which is a hypothesis of Theorem~\ref{thm:stability}. 
Strict admissibility of $(\tilde\scriptl,\tilde\be)$ is stable under small perturbations of $\tilde\be$.
The measures $|I_j(\by',s)|$ depend continuously on $\by',s$.
Therefore for any degree $\nu$, there exists a neighborhood $V$ of $0$ in $\Lambda_{d-1}$,
such that for all sufficiently small $s$ and all $\by'$ in this neighborhood,
$(|I_j(\by',s)|: j\in J)$ is strictly admissible relative to $\scriptl^1$
for all $(\by',s)$ in this neighborhood. 

The result for $d=1$ in \cite{christBLL} includes a genericity hypothesis that is not present in the
theory for $d>1$. $(\scriptl^1,\be)$ is said to be generic if for each extreme point $\bx$ of $\scriptk^1_\be$,
there are exactly $m$ indices $j\in J$ for which $|L_j^1(\bx)|=e_j/2$.
This is a property of $(\scriptl^1,\be)$ only; it is independent of the quantity $s$.
Denote by $\scriptk(\by')$ be the set of all $\bx\in\reals^m$
that satisfy $|L_j^d(\bx,\by')| \le e_j/2$ for all $j\in J$.
Equivalently, $|L_j^1(\bx)| \le (e_j^2/4 - |L_j^{d-1}(\by')|^2)^{1/2}$.
Thus $\scriptk(\by')$ is defined by replacing $e_j$ by $(e_j^2-4|L_j^{d-1}(\by')|^2)^{1/2}$
in its definition. Clearly, the set of all $\by'\in (\reals^{d-1})^J$ 
for which $\scriptk(\by')$ satisfies this definition of genericity,
is open and dense in a neighborhood of $0$, even though it may not contain $0$. 

Thus we may invoke the stability theorem of \cite{christBLL}
to conclude that for all $s$ sufficiently close to $0$,
\begin{equation}
\scriptt_{\scriptl^1}(\bI(\by',s))\leq\scriptt_{\scriptl^1} (\bI(\by',s)^\star)
-c\dist(\bI(\by',s),\scripto(\bI(\by',s)^\star))^2
\end{equation}
for a set $\Omega(s)$ of values of $\by'\in \Lambda_{d-1}$
whose $\lambda_{d-1}$ measure is greater than or equal to 
a positive constant that depends on $\nu,\scriptl,\be$
but may be taken to be independent of $s$.
Here $\scripto(\bI(\by',s))$ denotes the orbit under the translation action
of $\reals^m$; for $d=1$, the action of $\Sl(1)$ is trivial since 
this group consists only of the identity and the reflection about the origin,
and the latter preserves intervals centered at the origin.

The distance from $\bI(\by',s)$ to the orbit of $\bI(\by',s)^\star$ can be usefully 
reformulated in terms of the centers of the component intervals.
Denote by $c_j(x_j',s)$ the center of the interval $I_j(x_j',s)$, 
and let $\bc(\by',s)= (c_j(y'_j,s): j\in J)$.
As was shown in \cite{christRSult},
\begin{equation} \label{eq:cjPj}
c_j(x_j',s)=sP_j(x')+O(s^2),
\end{equation}
where $P_j$ is the polynomial associated to $G_j$ as in \eqref{eq:assocP}.

The distance from $(I_j(\by',s): j\in J)$ to the orbit of $(I_j(\by',s)^\star: j\in J)$
is comparable to $\dist((\bc(\by',s)),\Lambda_1)$.
Therefore
\begin{equation} \label{eq:gainbc}
\scriptt_{\scriptl^1}(\bI(\by',s))\leq\scriptt_{\scriptl^1} (\bI(\by',s)^\star)
-c\dist((\bc(\by',s)),\Lambda_1)^2,
\end{equation}
for all $\by'\in\Omega(s)$.
Combining \eqref{eq:gainbc} with \eqref{eq:cjPj} gives
\begin{equation}
\scriptt_{\scriptl^1}(\bI(\by',s))\leq\scriptt_{\scriptl^1} (\bI(\by',s)^\star)
-c s^2 P_\sharp(\bG)(\by')^2 + O(|s|^3)
\end{equation}
for all $\by'\in\Omega(s)$.

Since the measure of $\Omega(s)$ is bounded below uniformly
in $s$, and since $P_\sharp(\bG)^2$ is a polynomial of degree $\le 2\nu$
that does not vanish identically,
$\int_{\Omega(s)} P_\sharp(\bG)^2\,d\sigma$
is bounded below by a positive constant that is independent of $s$,
though it may depend on $\nu,d,\bG,\scriptl,\be$.

Therefore
\begin{align*}  
\scriptt_\scriptl(\bE(s)) 
&\le  \scriptt_\scriptl(\bEstar)
- \int_{\Lambda_{d-1}}
\Big[ \scriptt_{\scriptl^1}(\bI(\by',s)^\star) - \scriptt_{\scriptl^1}(\bI(\by',s)) \Big]
\,d\lambda_{d-1}(\by')
\\&
\le  \scriptt_\scriptl(\bEstar)
- \int_{\Omega(s)}
\Big[ \scriptt_{\scriptl^1}(\bI(\by',s)^\star) - \scriptt_{\scriptl^1}(\bI(\by',s)) \Big]
\,d\lambda_{d-1}(\by')
\\&
\le  \scriptt_\scriptl(\bEstar)
- cs^2 \int_{\Omega(s)}
P_\sharp(\bG)(\by',s)^2
\,d\lambda_{d-1}(\by') 
+ O(|s|^3)
\\&
\le  \scriptt_\scriptl(\bEstar)
- cs^2 
+ O(|s|^3)
\end{align*}
where $c$ depends on $P_\sharp(\bG)$ but is strictly positive.
Therefore the remainder term $O(|s|^3)$ can be absorbed,
yielding $\scriptt_\scriptl(\bE(s)) \le  \scriptt_\scriptl(\bEstar) - c's^2 $
for some $c'=c'(\bG)>0$,
as was to be shown.


\begin{thebibliography}{20}


\bibitem{bianchiegnell}
G.~Bianchi and H.~Egnell, 
{\em A note on the Sobolev inequality}, 
J. Funct. Anal. 100 (1991), no. 1, 18--24

\bibitem{blaschke}
W.~Blaschke,
{\em Kreis und Kugel}, Veit and Comp., Leipzig (1916)

\bibitem{BLL}
H.~J.~Brascamp, E.~H.~Lieb, and J.~M.~Luttinger,
{\em A general rearrangement inequality for multiple integrals}, 
J.\ Functional Analysis 17 (1974), 227--237

\bibitem{burchard} A.~Burchard,
{\em Cases of equality in the Riesz rearrangement inequality}, 
Ann. of Math. (2) 143 (1996), no. 3, 499--527








\bibitem{carlen2017}
E.~A.~Carlen, {\em Duality and stability for functional inequalities}, 
Ann. Fac. Sci. Toulouse Math. (6) 26 (2017), no. 2, 319--350

\bibitem{carlenfigalli}
E.~A.~Carlen, and A.~Figalli,
{\em Stability for a GNS inequality and the log-HLS inequality, 
with application to the critical mass Keller-Segel equation}, 
Duke Math. J. 162 (2013), no. 3, 579--625

\bibitem{carlenmaggi}
E.~Carlen and F.~Maggi,
{\em Stability for the Brunn-Minkowski and Riesz rearrangement inequalities, 
with applications to Gaussian concentration and finite range non-local isoperimetry}, 
Canad. J. Math. 69 (2017), no. 5, 1036--1063

\bibitem{christbmtwo}
M.~Christ,
{\em Near equality in the two-dimensional Brunn-Minkowski inequality},
preprint, math.CA arXiv:1206.1965

\bibitem{christbmhigh}
\bysame, {\em Near equality in the Brunn-Minkowski inequality},
preprint, math.CA arXiv:1206.1965


\bibitem{christgowersnorms}  \bysame,
{\em Subsets of Euclidean space with nearly maximal Gowers norms},
preprint, math.CA arXiv:1512.03355 

\bibitem{christRSult}  \bysame,
{\em A sharpened Riesz-Sobolev inequality}, 
preprint, math.CA arXiv:1706.02007

\bibitem{christBLL} \bysame,
{\em Equality in Brascamp-Lieb-Luttinger inequalities},
preprint, math.CA arXiv:1706.02778

\bibitem{christRBLLvariant} \bysame,
{\em A variant of the Rogers-Brascamp-Lieb-Luttinger inequality},
in preparation

\bibitem{christflock} M.~Christ and T.~Flock,
{\em Cases of equality in certain multilinear inequalities of Hardy-Riesz-Brascamp-Lieb-Luttinger type},
J. Funct. Anal. 267 (2014), no. 4, 998--1010



\bibitem{FJ1}
A.~Figalli and D.~Jerison,
{\em Quantitative stability for sumsets in $\reals^n$}, 
J. Eur. Math. Soc. (JEMS) 17 (2015), no. 5, 1079--1106 

\bibitem{FJ2}
\bysame,
{\em Quantitative stability of the Brunn-Minkowski inequality for sets of equal volume}, 
Chin. Ann. Math. Ser. B 38 (2017), no. 2, 393--412 

\bibitem{FJ3}
\bysame,
{\em Quantitative stability for the Brunn-Minkowski inequality}, 
Adv. Math. 314 (2017), 1--47

\bibitem{hardyetal}
G.~E.~Hardy, J.~E.~Littlewood, and G.~P\'olya,
{\em Inequalities}, Cambridge University
Press, London and New York (1952)


\bibitem{lieb1977}
E.~H.~Lieb, 
{\em Existence and uniqueness of the minimizing solution of Choquard's nonlinear equation}, 
Studies in Appl. Math. 57(1976/77), no. 2, 93--105

\bibitem{liebloss}
E.~Lieb and M.~Loss,
{\em Analysis}, Amer. Math. Soc., Providence, RI, 1997

\bibitem{riesz}
F.~Riesz,
{\em Sur une in\'egalit\'e int\'egrale}, J. London Math. Soc. 5 (1930), 162--168

\bibitem{rogers1}
C.~A.~Rogers, 
{\em Two integral inequalities}, 
J. London Math. Soc. 31 (1956), 235--238

\bibitem{rogers2}
\bysame,
{\em A single integral inequality}, 
J. London Math. Soc. 32 (1957), 102--108

\bibitem{sobolev}
S.~L.~Sobolev, {\em On a theorem of functional analysis}, 
Mat. Sb. (N.S.) 4 (1938), 471--497

\end{thebibliography}
\end{document}